\documentclass[a4paper,11pt]{amsart}

\usepackage[left=2.7cm,right=2.7cm,top=3.5cm,bottom=3cm]{geometry}
\usepackage{hyperref}

\usepackage{amsmath, amssymb, amsbsy, amscd, amsthm, amsfonts, mathrsfs} 
\usepackage{mathtools}

\usepackage{tikz}
\usetikzlibrary{matrix, arrows.meta}
\usetikzlibrary{intersections,calc,decorations.markings, decorations.text, arrows}

\usepackage{enumitem}

\setlist[itemize]{noitemsep}

\renewcommand{\leq}{\leqslant}
\renewcommand{\geq}{\geqslant}

\newtheorem{theo}{Theorem}[section]
\newtheorem{lemma}[theo]{Lemma}
\newtheorem{prop}[theo]{Proposition}
\newtheorem{cor}[theo]{Corollary}
\newtheorem{deff}[theo]{Definition}
\newtheorem{conj}[theo]{Conjecture} 
\theoremstyle{definition}
\newtheorem{rem}[theo]{Remark}
\newtheorem{example}[theo]{Example}

\numberwithin{equation}{section} 

\newcommand{\Q}{\mathbb{Q}}
\newcommand{\F}{\mathbb{F}}
\newcommand{\Z}{\mathbb{Z}}

\renewcommand{\P}{\mathbb{P}} 

\newcommand{\p}{\mathfrak{p}}
\newcommand{\q}{\mathfrak{q}}

\newcommand{\PGL}{\mathrm{PGL}}
\newcommand{\Res}{\mathrm{Res}} 

\newcommand{\Per}{\mathrm{Per}} 
\newcommand{\PrePer}{\mathrm{PrePer}} 
\newcommand{\Fix}{\mathrm{Fix}} 
\newcommand{\Tail}{\mathrm{Tail}} 

\title[Cycles for rational maps over function fields]{Cycles for rational maps over global function fields with one prime of bad reduction}
\author{Silvia Fabiani}

\address{{\sc Silvia Fabiani}: Universit\`a di Pisa}

\email{s.fabiani4@studenti.unipi.it}

\begin{document}


\begin{abstract}
	Let $K$ be a global function field of characteristic $p$ and degree $D$ over $\F_{p}(t)$. We consider dynamical systems over the projective line $\P^1(K)$ defined by rational maps with at most one prime of bad reduction. The main result is an optimal bound for cycle lengths that only depends on $p$ and $D$. A bound for the cardinality of finite orbits is given as well. Our method is based on a careful analysis (for every prime of good reduction) of the $\p$-adic distances between points belonging to the same finite orbit, in part motivated by previous work by Canci and Paladino. Valuable insight is provided by a certain family of polynomials. In this case we also gain a good deal of information about the structure and size of the set of periodic points for polynomials of given degree.
\end{abstract}

\maketitle


\section{Introduction}

Let $\phi \colon \P^N \rightarrow \P^N$ be an endomorphism of degree $d \geqslant 2$ defined over a global field $K$. We denote the $n$-fold composition $\phi \circ \dots \circ \phi$ with $\phi^n$. The \textit{forward orbit} of some point $P$ is the set $\{\phi^n(P) : n \geqslant 0\}$. One of the main problems in discrete dynamics is to classify points according to the behavior of their orbits. We say that $P$ is a \textit{fixed point} for $\phi$ if $\phi(P)=P$ and $P$ is \textit{periodic} if there exists a positive integer $n$ such that $\phi^n(P)= P$. The minimal $n$ satisfying the previous equality is called the \textit{minimal period} or \textit{exact period} of $P$. When $P$ is periodic we refer to its forward orbit as a \textit{cycle} and we call its cardinality \textit{length}. A \textit{preperiodic} point is a point with finite forward orbit, or equivalently a point whose orbit contains a periodic point. The set of periodic points for $\phi$ belonging to $\P^N(K)$ is denoted by $\Per(\phi, K)$. Similarly $\Fix(\phi, K)$ and $\PrePer(\phi, K)$ stand for the set of $K$-rational fixed and preperiodic points respectively.

In 1950 Northcott proved that the set of preperiodic points for a rational map over a global field has bounded height and hence it is finite (see \cite{Nor50}). Unfortunately the bound established by Northcott, though effective, depends on the coefficients of the rational map. In 1994 Morton and Silverman proposed the following conjecture (see \cite{MS94}).
\begin{conj}[Uniform Boundedness Conjecture]
Fix integers $N \geqslant 1$, $D \geqslant 1$ and $d \geqslant 2$. There exists a constant $B(N, D, d)$ such that for every number field $K$ of degree $D$ over $\Q$ and for all degree $d$ morphisms $\phi \colon \P^N \rightarrow \P^N$ defined over $K$, the number of $K$-rational preperiodic points is uniformly bounded:
\begin{equation*}
	\# \PrePer(\phi, K) \leqslant B(N, D, d).
\end{equation*} 
\end{conj}
The above statement (UBC for short) is a very bold conjecture: as an example we can show that the case $N= 1$ and $d=4$ implies Merel's theorem (see \cite{Merel96}), i.e. the $K$-rational torsion points of elliptic curves defined over $K$ are bounded solely in terms of the degree $D = [K: \Q]$. Consider the multiplication by two map $[2] \colon E \rightarrow E$ on the elliptic curve $E$, then the projection on the first coordinate (here denoted by $x$) induces a rational map of degree $4$ over $\P^1(K)$ known as Latt\`es map:

\begin{center}
	\begin{tikzpicture}
	\matrix(m)[matrix of math nodes,row sep=2em,column sep=3em]
		{	E & E\\
			\P^1 & \P^1\!.\\};
	\path[->]
		(m-1-1) edge node [left] {$x$} (m-2-1)
		edge node [above] {$[2]$} (m-1-2)
		(m-2-1) edge node [below] {$\phi_{E, 2}$} (m-2-2)
		(m-1-2) edge node [right] {$x$} (m-2-2);
	\end{tikzpicture}
\end{center}

It is not difficult to see that $K$-rational torsion points of $E$ correspond to preperiodic points for $\phi_{E, 2}$: more precisely \[x(E(K)_{\mathrm{tors}}) = \PrePer(\phi_{E, 2}, K).\]
Then uniform boundedness of torsion points would be an immediate consequence of uniform boundedness of preperiodic points. Of course, running the converse argument, Merel's theorem implies uniform boundedness for the collection of Latt\`es maps (see Section 6.6 in \cite{Sil07} for the general definition). Aside from monomials and Chebyshev polynomials (both arising from endomorphisms of the multiplicative group $K^*$), that of Latt\`es maps is the only nontrivial family of rational maps for which the UBC is known to be true. Thus the conjecture seems very far from being proved, despite growing evidence that it is valid at least for quadratic polynomials over $\Q$ (see \cite{Mor98}, \cite{FPS97} and \cite{Poonen98} for a handful of results). 

In a recently published paper (see \cite{DP20})
Doyle and Poonen prove the function field analogue of uniform boundedness for polynomials of the form $z^d + c$. 
\begin{theo}[Doyle and Poonen, 2020]
	Fix $D \geqslant 1$, $d \geqslant 2$ and a field $k$ such that $\mathrm{char}(k)$ does not divide $d$. Let $K$ be the function field of an integral curve over $k$. There exists a constant $B= B(D, d, K)$ such that for every extension $L/K$ of degree $D$ and every $c \in L$ not algebraic over $k$ the number of preperiodic points in $L$ for the polynomial $z^d + c$ is uniformly bounded:
\begin{equation*}
	\# \PrePer(z^d + c,\, L) \leqslant B(D, d, K).
\end{equation*}
Moreover, if $k$ is finite the same is true even for $c$ algebraic over $k$.
\end{theo}
This result is achieved as a consequence of the growth of gonalities of dynatomic curves attached to those polynomials. This approach has the additional bonus of proving the remarkable fact (valid for polynomials $z^d + c$ over number fields) that uniform boundedness of cycle lengths implies uniform boundedness of preperiodic points. The last statement generalizes a previous theorem by Poonen (see \cite{Poonen98}) which asserts that if for every $c \in \Q$ the set $\Per(z^2 + c,\, \Q)$ contains no points of minimal period strictly larger than $3$, then $\PrePer(z^2 + c,\, \Q)$ has cardinality at most $9$.

Some mathematicians have undertaken the comparatively easier task of determining a bound for preperiodic points that depends, although weakly, on certain properties of the map $\phi$. The key ingredient here is good reduction (see Subsection \ref{subs:goodred} for details). The goal of this research is a uniform bound depending on the number of primes of bad reduction for the map $\phi$ in addition to the usual $N, D$ and $d$. For $N=1$ there are several results providing uniform bounds for periodic or preperiodic points of maps with a fixed number $s$ of bad primes. We briefly recall a few of them.

Benedetto in \cite{Ben07} proved a bound for the number of preperiodic points of polynomials over global fields that is of the order $\mathrm{O}(s \log s )$ where the big-$\mathrm O$ constant for large $s$ is essentially $(d^2 - 2d +2) / \log d$. The proof is based on the study of the filled Julia sets at every prime, thus involving techniques of both complex and non-archimedean dynamics.

Canci, Troncoso e Vishkautsan in \cite{CTV19} give a bound for preperiodic points of rational functions over number fields that is quadratic in $d$ and exponential in $s$.\footnote{There is some inconsistency in the literature about the definition of \textit{bad prime}. Some authors characterize bad reduction as absence of potentially good reduction as in \cite{Ben07}, while others (e.g. \cite{CTV19}) say that a map has bad reduction at a prime if it has not (simple) good reduction. In this paper we will adopt the second point of view: see Definition \ref{def_goodred}.} This result depends on a reduction to unit equations with a number of solutions uniformly bounded by \cite[Theorem 1.1]{ESS02}. Similar techniques have been applied to find bounds for the cardinality of finite orbits in terms of the number of bad primes, even for global function fields, as in \cite{CP16}. 

In this paper we focus our attention on rational maps defined over a global function field with bad reduction at just one prime. We are going to study cycle lengths and the cardinality of finite orbits: in our setting we are able to find much better bounds than those obtained by specializing the results in \cite{CP16} to $s=1$. We remark that our bound for cycle lengths is optimal (see Example \ref{ExOptCycLenght}). We summarize these results (contained in Corollary \ref{Cbound0} and Theorem \ref{TOrbits}) in the following theorem, which improves previous work by Canci and Paladino (see \cite[Theorem 1.2]{CP16b}).
 
\begin{theo}\label{Tbound}
    Let $\phi \colon \P^1\rightarrow \P^1$ be a rational map defined over a global function field $K$ with bad reduction at one prime. Let $p$ be the characteristic of the field and $D$ the degree of the extension $K/\F_p(t)$.
    \begin{enumerate}
    	\item If $\mathcal C \subset \P^1(K) $ is a cycle for $\phi$, then $\# \mathcal C\leqslant p^D + 1$.
    	\item If $\mathcal A \subset \P^1(K) $ is a finite orbit for $\phi$, then $\# \mathcal A \leq 3p^D + 6$.
    \end{enumerate}
\end{theo}


We conclude with a brief overview of the contents of this paper. 
In Section \ref{Spol} we consider the particular case of the family of polynomials with coefficients in a ring of integers of a global function field and invertible leading coefficient.
In this setting we are able to prove that the length of a cycle cannot exceed the cardinality of the full constant field (Theorem \ref{Tpol}) and show that this result is optimal (Example \ref{ExOptCycLenght0}). In addition we give a bound for the number of periodic points for polynomials of degree $d$ that is linear in $d$ (Theorem \ref{Tpol3}). We also remark that this last result is in some sense \textit{good}, by showing that the number of periodic points cannot be bounded independently of $d$ (Example \ref{ExDepOnd}). After that we turn to rational functions with one prime of bad reduction: the relevant definitions are given in Sections \ref{S3} and \ref{preliminaries}.
Interestingly enough, some of the properties of the polynomials considered in Section \ref{Spol} can be generalized to rational maps, the main difference being that the proofs will be more involved and the results (generally speaking) way less effective. This analogy is explored in Sections \ref{Sproof} and \ref{Sdbound}. We prove (Theorem \ref{T1}) that the $\p$-adic distance between any two distinct points belonging to the same cycle is constant for every fixed prime of good reduction $\p$. Similar statements (proven in Section \ref{Sfinite}) are true for finite orbits and these results are the main step towards the proof of Theorem \ref{Tbound}. Further applications are given in Section \ref{Sdbound}.


\section{Polynomials over global function fields}\label{Spol}

Let $K$ be a global function field of characteristic $p$. We fix one prime at infinity and denote with $\mathcal{O}$ its ring of integers, i.e. the ring of rational functions regular outside $\infty$. The full constant field of $K$ is the algebraic closure of $\F_p$ in $K$, hence a finite field $\F_q$, and the group of units $\mathcal{O}^*$ coincides with $\F_{q}^*$. The aim of this section is to study some dynamical properties of polynomials in $\mathcal{O}[X]$ with leading coefficient in $\mathcal{O}^*$. These polynomials are far more manageable than rational functions and the proofs do not require specific tools. Nonetheless we shall see that some of the techniques can be generalized to rational maps with exactly one prime of bad reduction.

In this section only we will consider polynomials as maps over $K$ instead of $\P^1(K)$, because the point at infinity is obviously fixed.
Our first result is a bound for cycle lengths: the argument is not new and can also be found in \cite[Section 2.2]{Zieve96} in the more general context of polynomials defined over integral domains. However, to the author's knowledge, it has not been applied to this particular setting yet.

We begin by showing that linear polynomials are dynamically uninteresting, since all points are periodic.
\begin{lemma}\label{Lpol}
    If $\phi: K \rightarrow K$ is a linear polynomial in $\mathcal{O}[X]$ with leading coefficient in $\mathcal{O^*}$, then there exists $n \leqslant q$ (actually $\leqslant q-1$ if $q>p$) such that $\phi^n(X)=X$ as polynomials.
\end{lemma}

\begin{proof}
	If $\phi$ is monic and linear, then $\phi^p(X)=X$. Next assume $\phi(X)= wX + a$ with $w \neq 1$ and let $n$ be the multiplicative order of $w$ in $\F_q^*$. Then
\begin{equation*}
	\phi^n(X)=w^nX+ (1+w+ \dots + w^{n-1})a= X.\qedhere
\end{equation*} 
\end{proof}

\begin{theo}\label{Tpol}
	If $\phi \colon K \rightarrow K$ is a polynomial in $\mathcal{O}[X]$ with leading coefficient in $\mathcal{O}^*$ and $\deg(\phi)\geqslant 2$, then all $K$-periodic points are in $\mathcal{O}$ and have minimal period at most $q$.
\end{theo}

\begin{proof}
    Let $x\in K$ be a periodic point of exact period $n$, then it satisfies $\phi^n(x)- x= 0$ and thus it is a root of a monic polynomial with coefficients in $\mathcal{O}$. Then $x$ is in $\mathcal{O}$ and so are all the points in its forward orbit $\phi^i(x)$ for $i\geqslant 0$ (we put $\phi^0(x) = x$). Moreover for any distinct $x, y \in \mathcal{O}$ one has $(y-x)\mathcal{O} \supseteq (\phi(y) - \phi(x)) \mathcal{O}$, in particular:
\begin{equation*}
	(\phi(x) - x)\mathcal{O} \supseteq (\phi^2(x) - \phi(x)) \mathcal{O} \supseteq \dots \supseteq (x - \phi^{n-1}(x)) \mathcal{O} \supseteq (\phi(x) -x ) \mathcal{O}.
\end{equation*}
	Thus the above ideals coincide and for every $i =1,\dots,n$ there exists $u_i \in \mathcal{O}^*= \F_q^*$ such that $\phi^i(x) - \phi^{i-1}(x)= u_i (\phi(x) - x)$ . Then for every index $i$ 
	\begin{equation*}
	\phi^i(x) - x= (\phi^i(x) - \phi^{i-1}(x))+ (\phi^{i-1}(x)- \phi^{i-2}(x))+ \dots + (\phi(x) - x)=w_i(\phi(x) - x)
	\end{equation*}
	where
	\begin{equation*}
	w_i= \dfrac{\phi^i(x) - x}{\phi(x) - x}, \quad \text{for } i = 1, \dotsc, n
	\end{equation*}
	are distinct elements in $\F_q$, hence the thesis.
\end{proof}

For instance, every monic polynomial with coefficients in $\F_q[t]$ has cycles of length at most $q$ in $\F_q(t)$. This follows from Theorem \ref{Tpol} by putting $K = \F_q(t)$ and $\infty = \frac{1}{t}$, i.e. the prime attached to the valuation ${v}(f(t)/g(t))=\deg(g) - \deg (f)$.

\begin{rem}\label{Rqp^D}
	It is worth mentioning that the previous statement yields a bound in terms of $p$ and the degree of the extension $K/\F_p(t)$. Indeed, let $D=[K: \F_p(t)]$, then $\F_q\subseteq K$ yields $n\leqslant q\leqslant p^D$ in the estimate.
\end{rem}
The bound given in Theorem \ref{Tpol} is sharp for every global function field. More precisely, there exist monic polynomials defined over the finite field $\F_q$ with a unique cycle of length $q$. We may obviously consider these as dynamical systems defined over any field $K$ with $\F_q$ as full constant field.

\begin{example}\label{ExOptCycLenght0}
    Let $f \colon \F_q \rightarrow \F_q$ be any map, then $f$ is induced by (exactly) one polynomial $\phi \in \F_q[X]$ of degree smaller than $q$. Indeed, the obvious map 
	$\F_q[X] \rightarrow \{\text{maps $\F_q\rightarrow \F_q$}\}$ induces a ring isomorphism $\F_q[X]/(X^q-X) \cong \{\text{maps $\F_q\rightarrow \F_q$}\}$. Then for a fixed $f$ we may take $\phi$ to be the unique representative of $f$ of degree less than $q$.
	Observe (this will be useful in Section \ref{Sproof}) that we can modify $\phi$ to get a monic polynomial of any degree $d \geq q$ by simply adding a monic multiple of $X^q-X$. 
	In particular there exist monic polynomials of assigned degree $d \geqslant q$ with periodic orbit
	\[w_0 \mapsto w_1 \mapsto \cdots \mapsto w_{q-1} \mapsto w_0\]
	where the $w_i$'s are the elements of $\F_q$. 
\end{example}

As previously remarked, linear polynomials are not of much interest (from a dynamical viewpoint) because all points are periodic. From now on we will then assume $\deg(\phi) \geqslant 2$. We have another application of the ideas contained in the proof of Theorem \ref{Tpol}.

\begin{theo}\label{Tpol3}
    If $\phi \in \mathcal{O}[X]$ is a polynomial of degree $d \geqslant 2$ with leading coefficient in $\mathcal{O}^*$, then $\# \Per(\phi, K) \leqslant (q-1) (d-1) + 1.$ 
\end{theo}

\begin{proof}
	By Theorem \ref{Tpol} all periodic points belong to $\mathcal O$. Fix two distinct periodic points $x$ and $y$ and take a minimal $n$ such that both points are periodic of period $n$. Then
\begin{equation*}
(x - y)\mathcal{O} \supseteq (\phi(x) - \phi(y)) \mathcal O \supseteq \cdots \supseteq (\phi^{n-1}(x) - \phi^{n-1}(y)) \mathcal{O} \supseteq (x - y)\mathcal{O}
\end{equation*}
as in the proof of Theorem \ref{Tpol}.
Then $x$ is a solution of one of the $q-1$ equations of degree $d$
\begin{equation*}
\phi(X) - \phi(y) = u (X - y) \quad \text{with $u \in \F_q^*$}
\end{equation*}
and $y$ satisfies all of them. The bound follows immediately.
\end{proof}

\begin{rem}
The proof of Theorem \ref{Tpol3} works for polynomials with coefficients in any integrally closed domain $R$,
provided $R^*$ is finite. The bounds have to be modified in the obvious way replacing $q-1$ with $\#R^*$. The situation is more complicated for Theorem \ref{Tpol}, where $R^*$ needs to be the group of units of a finite field. 
\end{rem}

The key point in the proof of Theorem \ref{Tpol} is that if $x$ is periodic of minimal period $n$ then for every $i=1, \dots, n-1$ 
\begin{equation}\label{eqpol}
\phi^i (x) -x=u_i(\phi(x)-x) \quad \text{for some $u_i \in \F_q^*$.}
\end{equation}
It follows that the difference between any two distinct points in the cycle of $x$ is equal to $\phi(x) - x$ up to multiplication by a unit. This remark partly motivates the following proposition.

\begin{prop}\label{Ppol}
	Let $\phi \in K[X]$ be a polynomial of degree $d$.
	Assume that $\mathcal P$ and $\mathcal P'$ are subsets of $K$ satisfying:
	\begin{enumerate}[label=\textnormal{(\alph{enumi})}]
	 \item the differences $y - x$ with $(x, y)$ varying over the pairs of distinct elements in $\mathcal{P}$ are all equal up to multiplication by elements of $\F_q^*$;
	 \item $\mathcal P' \cup \phi(\mathcal P') \subseteq \mathcal P$.
	\end{enumerate}
Then the following properties hold:
	\begin{enumerate}
		\item $\# \mathcal{P} \leq q$ (only \textnormal{(a)} is needed for this);
		\item either $\# \mathcal{P'} \leqslant d$ or $\phi$ is conjugated over $K$ to a polynomial with coefficients in $\F_q$.
	\end{enumerate} 
\end{prop}

\begin{proof}
	For the first part, let $x_0, \dotsc, x_{n-1}$ be $n$ distinct points in $\mathcal{P}$, then for $i=1, \dotsc , n-1$
	\begin{equation}\label{equnit}
	x_i-x_0= u_i(x_1-x_0) \quad \text{with $u_i \in \F_q^*$ }
	\end{equation}
	and the $u_i$'s are distinct elements of $\F_q^*$, thus proving (a).
	
	Now let $\mathcal{P}= \{x_0, \dotsc , x_{n-1}\}$ and consider the affine invertible map over $K$ defined by 
	\[\eta(X)= \dfrac{X-x_0}{x_1 - x_0}\,,\]
	then by \eqref{equnit} the image $\eta(\mathcal{P})$ is contained in $\F_q$. Next assume $\mathcal{P'} \subseteq \mathcal P$ contains at least $d+1$ elements and $\phi(\mathcal{P'}) \subseteq \mathcal{P}$; without loss of generality we may assume $\mathcal P'= \{x_0, \dotsc, x_d\}$. Consider the polynomial $\psi 
	= \eta\circ\phi\circ\eta^{-1}\in K[X]$ which still has degree $d$ and satisfies
	\[ \psi(u_i)=\dfrac{\phi(x_i)-x_0}{x_1-x_0} = y_i\in \F_q\,.\]
Write $\psi(X)=\sum_{i=0}^d a_i X^i$, then the vector $(a_0, \dotsc, a_d)$ is a solution of the linear system with invertible Vandermonde matrix

\begin{equation*}
\left[\begin{array}{cccc}
1 & u_0 & \cdots & u_0^d\\
1 & u_1 & \cdots & u_1^d\\
\vdots & \vdots & &	\vdots\\
1 & u_d & \cdots & u_d^d\\
\end{array}\right]
\left[\begin{array}{c}
a_0\\
a_1\\
\vdots\\
a_d
\end{array}\right]=
\left[\begin{array}{c}
y_0 \\
y_1\\
\vdots\\
y_d
\end{array}\right].
\end{equation*}
	Moreover both the vector $(y_0, \dotsc , y_d)$ and the matrix have entries in the finite field $\F_q$, therefore the system admits a unique solution in $\F_q^{d+1}$ corresponding to $\psi \in \F_q[X]$.
\end{proof}

We will always apply Proposition \ref{Ppol} with $\mathcal P = \mathcal P'$ a $\phi$-invariant set, i.e. a set satisfying $\phi(\mathcal P ) \subseteq \mathcal P$.
Now let $x, y$ be elements of $K$; to ease notation we will write $x \sim y$ if $y = ux$ for some $u \in \mathcal{O}^*$. Then it is not difficult to see that
\begin{enumerate}
    \item $\sim$ is an equivalence relation;
    \item if $x,y,z$ are distinct elements of $K$, then $y-x \sim z - x$ implies $y-x \sim y-z$.
\end{enumerate}   
It immediately follows that for a subset $\mathcal{P}$ of $K$ the following properties are equivalent:
\begin{enumerate}
	\item [(a)] the differences $y - x$ with $(x, y)$ varying over the pairs of distinct elements in $\mathcal{P}$ are all equal up to multiplication by units;
	\item [(a')]fix $y \in \mathcal{P}$, then the differences $y-x$ with $x \neq y$ varying over the elements of $\mathcal{P}$ are all equal up to multiplication by units.
\end{enumerate}
When needed, we will always check condition (a') instead of (a).

\begin{rem}\label{Rcycle}
	Let $\phi \in \mathcal O[X]$ be a polynomial of degree $d \geq 2$ with invertible leading coefficient, then by equation \eqref{eqpol} we may apply Proposition \ref{Ppol} to a cycle $\mathcal{P}$. Then the length of the cycle is at most $d$ unless $\phi$ is conjugated to some polynomial defined over $\F_q$. In that case (after a suitable change of coordinates) $\PrePer(\phi, K)$ equals $\F_q$: indeed every point in $\F_q$ is preperiodic because $\F_q$ is finite and $\phi$-invariant. Conversely, if $x \in K$ is preperiodic for $\phi$ then it solves an equation of the form $\phi^{m+n}(X)= \phi^m(X)$ and therefore it is algebraic over $\F_q$, hence $x \in \F_q$.
\end{rem}

Now the general idea is to find larger subsets of $\Per(\phi, K)$ for which both conditions of Proposition \ref{Ppol} hold. In Lemmas \ref{Lp1} and \ref{Lp2}, $\phi$ is a polynomial in $\mathcal O[X]$ of degree $d \geq 2$ with leading coefficient in $\mathcal O^*$.

\begin{lemma}\label{Lp1}
	Assume that $\phi$ has at least one fixed point, then for every nontrivial cycle $\mathcal{C}$ the set $\mathcal{P}=\Fix(\phi, K) \cup \mathcal{C}$ satisfies both hypotheses of Proposition \ref{Ppol}.
\end{lemma}

\begin{proof}
	Since $\mathcal{P}$ is clearly a $\phi$-invariant set, we only need to show that condition (a') holds. Let $y$ be a fixed point for $\phi$ and take $x \in \mathcal{C}$ of exact period $n>1$, then as in the proof of Theorem \ref{Tpol} 
	\begin{equation*}
	(y-x) \mathcal{O} \supseteq (y - \phi(x)) \mathcal{O} \supseteq \dots \supseteq (y - \phi^{n-1}(x))\mathcal{O} \supseteq (y-x) \mathcal{O}.
	\end{equation*}
	Then $y -x \sim y - \phi^i(x)$ for every $i$, which means that the thesis is true for the subset $\{y\} \cup \mathcal{C}$. In particular $y- x \sim \phi(x) - x$ and this holds for every fixed point. Now let $y' \neq y$ be another fixed point for $\phi$ (if there exists one), then $y - x \sim y' - x$ because both are equivalent to $\phi(x)- x$. This in turn implies $y - x \sim y - y'$.
\end{proof}
		
\begin{lemma}\label{Lp2}
	If the set of periodic points contains at least two fixed points and a cycle, then it satisfies both hypotheses of Proposition \ref{Ppol}.
\end{lemma}

\begin{proof}
To begin with, $\Per(\phi, K)$ is obviously $\phi$-invariant. As seen in the proof of Lemma \ref{Lp1} the distance between any two fixed points $y \neq y'$ satisfies $y - y' \sim y - x$ for every $x$ that belongs to a cycle. This simple remark includes all possible cases for $y - x$ with $x \neq y$ varying over the set of periodic points.
\end{proof}

\begin{theo}\label{Tpol2}
    Let $\phi \in \mathcal{O}[X]$ be a polynomial of degree $d \geqslant 2$ with leading coefficient in $\mathcal{O}^*$. Then $\# \Per(\phi, K) \leqslant \min \{d,q\}$ unless one of the following is satisfied:
	\begin{enumerate}[label=\textnormal{(\alph{enumi})}]
	\item $\phi$ is conjugated to a polynomial over $\F_q$, then $\# \PrePer(\phi, K)= q$;
	\item $\Per(\phi,K)=\Fix(\phi,K)$, then $\# \Per(\phi, K) \leqslant d$.
	\item $\Per(\phi, K)$ is the union of at least two $n$-cycles (with $n > 1$) and at most one fixed point.
	\end{enumerate}
\end{theo}

\begin{proof}
	The previous results immediately yield the following cases:
	\begin{enumerate}
		\item if every periodic point is a fixed point, then $\# \Per(\phi,K) \leqslant d$;
		\item if the set of periodic points consists of a unique $n$-cycle, then $\Per(\phi, K)$ satisfies the hypotheses of Proposition \ref{Ppol};
		\item if the set of periodic points is the union of one fixed point and one $n$-cycle, then by Lemma \ref{Lp1} we may apply Proposition \ref{Ppol} to $\mathcal{P}= \Per(\phi, K)$;
		\item if $\Per(\phi, K)$ contains at least two fixed points and a cycle, then again the hypotheses of Proposition \ref{Ppol} hold for  $\mathcal{P}= \Per(\phi, K)$ by Lemma \ref{Lp2}.	
	\end{enumerate}
We show that the remaining cases can somehow be reduced to (iv). Indeed, suppose that the set $\mathcal{P}=\Per(\phi, K)$ contains cycles of at least two distinct lengths $1 <m<n$. Let $\psi= \phi^{m}$, then of course a point of $K$ is periodic for $\phi$ if and only if it is periodic for $\psi$, i.e. $\mathcal{P}= \Per(\psi, K)$. Moreover the set $\mathcal{P}$ contains at least $m$ fixed points for $\psi$: the points of minimal period $m$ for $\phi$. Then again we can apply Lemma \ref{Lp2} and $\mathcal{P}$ satisfies the conditions in Proposition \ref{Ppol} (for both $\psi$ and $\phi$). Hence $\mathcal{P}$ has cardinality at most $d$ unless $\phi$ is conjugated to some polynomial over $\F_q$. In the latter case, after a change of coordinates the set of preperiodic points is $\F_q$ (see Remark \ref{Rcycle}).
\end{proof}

By Theorem \ref{Tpol2} the set of periodic points has cardinality at most $q$, unless it coincides with $\Fix(\phi, K)$ or it is the union of a certain number of cycles of the same length and at most one fixed point. We provide counterexamples for these two cases, taking $K = \F_q(t)$ and $\infty=\frac{1}{t}$, so that the ring of integers is $\F_q[t]$.

\begin{example}\label{ExDepOnd}
   	Take distinct elements $f_1, \dotsc , f_d \in \F_q[t]$ and define
	\begin{equation*}
	\psi_d(X) = \prod_{i=1}^{d}(X - f_i) + X.
	\end{equation*}
	Then $f_1, \dotsc f_d$ are fixed points for $\psi_d$. Moreover, by Theorem \ref{Tpol2}, $\Per(\psi_d, K) = \Fix(\psi_d, K)$ for every $d>q$.
\end{example}
	
\begin{example}\label{ExDepOnd'}
    We refine the previous example in order to obtain a polynomial with a fixed point and several $n$-cycles.
	Take $w \in \F_q^*$ of multiplicative order $n>1$ and $f_1, \dotsc, f_m$ nonzero elements of $\F_q[t]$ such that their $n$-th powers $f_1^n, \dotsc, f_m^n$ are distinct. Then for every $i= 1, \dotsc, m$ the polynomial $X^n - f_i^n$ is separable with roots $f_i, wf_i, \dotsc, w^{n-1} f_i$. Now define the polynomial of degree $d = mn + 1$
	\begin{equation*}
	\phi_m(X)= X \prod_{i=1}^{m} (X^n - f_i^n) + wX.
	\end{equation*}
	Then $0$ is fixed and the set of periodic points contains at least $m$ distinct cycles of length $n$: for $i= 1, \dotsc, m$
	\begin{equation*}
	f_i \mapsto wf_i \mapsto \cdots \mapsto w^{n-1} f_i \mapsto f_i\, ,
	\end{equation*}
    so $\Per(\phi_m, K)$ has cardinality at least $d$. Of course $m$ (and hence $d$) can be arbitrarily large.
	Also note that, by Theorem \ref{Tpol2}, the polynomial $\phi_m$ has as at most one fixed point in $K$ for large $m$, i.e. $0$ is the unique root of $\phi_m(X) - X$ in $\F_q(t)$.

\end{example}

\begin{rem}\label{Rdbound}In case (c) of Theorem \ref{Tpol2} we can still prove (as a consequence of Proposition \ref{Ppol} and Lemma \ref{Lp1}):
	\begin{enumerate}
		\item if $\Per(\phi, K)$ is the union of cycles of length $n$, then either $n \leqslant d$ or $\phi$ is conjugated to a polynomial over $\F_q$;
		\item if $\Per(\phi, K)$ is the union of exactly one fixed point and some cycles of length $n$, then either $n+1 \leqslant d$ or $\phi$ is conjugated to a polynomial over $\F_q$.
		\end{enumerate}
	With the same assumptions one also finds $n \leqslant q$ in case (i) and $n+1 \leqslant q $ in case (ii).
\end{rem}

\begin{example}
We are now able to characterize all possible sets of periodic points for polynomials of degree two that are not defined over the full constant field. The length of a cycle is at most two by Remark \ref{Rcycle}, moreover $2$-cycles and fixed points cannot coexist by Remark \ref{Rdbound}. Therefore $\Per(\phi, K)$ can either be empty, consist of at most two fixed points, or be the union of $2$-cycles. Periodic points of period two are roots of $\phi^2(X) - X$ (at most four) and at least one of them is a fixed point (possibly contained in some proper extension of $K$). Then $\Per(\phi, K)$ contains at most one $2$-cycle. Examples \ref{ExDepOnd} and \ref{ExDepOnd'} show that there exist polynomials of degree two with one or two fixed points or a unique $2$-cycle. 
\end{example}


\section{Notation and background}\label{S3}

The aim of this section is to provide the reader with the basics about reduction of rational maps modulo a prime and good reduction. We will then define a metric induced by the $\p$-adic valuation over $\P^1(K)$ and recall some of its properties which are relevant when studying dynamical questions. A more detailed discussion of these topics can be found in \cite[Chapter 2]{Sil07}. We conclude this section with a result (Lemma \ref{L1new}) that is crucial for the proof of our bounds for rational functions.

In this section and at the beginning of Section \ref{Sfinite} we will be using the following notation:
\[\begin{array}{cl}
K &\text{a valued field with discrete valuation $v=v_{\p}\colon K^* \rightarrow \Z$;}\\
R &=\{\alpha \in K : v(\alpha) \geqslant 0 \}\text{ the discrete valuation ring attached to $v$;}\\
R^* &=\{ \alpha \in K : v(\alpha)=0\}\text{ the group of units of $R$;}\\
\p &=\{ \alpha \in K : v(\alpha) \geqslant 1\}\text{ the maximal ideal of $R$;}\\
k(\p) &=R/\p\text{ the residue field of $R$;}\\
\widetilde{\;} &\text{reduction modulo $\p$, i.e. $R \rightarrow k(\p)$, $\alpha \mapsto \tilde{\alpha}$;}\\
\phi & \text{a rational map } \mathbb{P}^1(K)\rightarrow \mathbb{P}^1(K) \text{ defined over } K.\\
\end{array}\]
We remark that most results in \cite[Chapter 2]{Sil07} are stated for local fields but the same proofs work for our valued field/discrete valuation ring setting as well.


\subsection{Rational maps with good reduction}\label{subs:goodred}
We begin with some definitions.

\begin{deff}
	A point $ P = [x : y] \in \P^1(K)$ in homogeneous coordinates is in \textbf{normalized form} with respect to $\p$ if 
	\[\min\{v_{\p}(x), v_{\p}(y)\} = 0,\]
	i.e. at least one of $x, y$ belongs to $R^*$.
    If $ P=[x: y] \in \P^1(K)$ is normalized with respect to $\p$, the \textbf{reduction of $P$ modulo $\p$} is the point $\widetilde{P}= [\tilde{x}: \tilde{y}] \in \P^1(k(\p))$.
\end{deff}

\begin{deff} Write a rational map $\phi$ as
	\[\phi([X:Y])=[F(X,Y): G(X,Y)]\]
	with $F, G$ homogeneous polynomials in $K[X, Y]$. We say that the pair $[F:G]$ is in \textbf{normalized form} with respect to $\p$ if $F$ and $G$ are polynomials in $R[X,Y]$ and at least one of them does not belong to $\p R[X,Y]$. If $\phi=[F:G]$ is in normalized form, the \textbf{reduction of $\phi$ modulo $\p$}, denoted by $\widetilde{\phi}$, is defined reducing all the coefficients of $F$ and $G$ modulo $\p$. 
\end{deff}

Every point can be written in normalized form with respect to $\p$, though not in a unique way: indeed, if $[x : y]$ is normalized, so is $[ux : uy]$ for every $u \in R^*$. Anyway the reduction modulo $\p$ does not depend on the choice of coordinates by \cite[Proposition 2.7]{Sil07} and the same applies to rational maps. In general however the reduction of maps does not behave nicely: for instance it may not be true that $\widetilde{\phi}(\widetilde{P})= \widetilde{\phi(P)}$, nor that $\widetilde{\phi} \circ \widetilde{\psi}= \widetilde{\phi \circ \psi}.$ The following fact is responsible for this: even though $\phi=[F:G]$ is given by a pair of relatively prime homogeneous polynomials, the reduced polynomials may acquire some non-trivial common roots in an algebraic closure of the residue field $k(\p)$. The notion of good reduction was introduced to rule out this kind of behavior.

\begin{deff} Let $\phi=[F:G]$ be written in normalized form with respect to $\p$. The \textbf{resultant} $\Res(\phi)$ is defined as the homogeneous resultant $\Res(F,G)$ (see \cite[Section 2.4]{Sil07}).
\end{deff}

\begin{rem}
	Since the pair $[F:G]$ is unique up to multiplication of both terms by a unit $u \in R^*$, the quantity $v_{\p}(\Res(\phi))$ depends only on the map $\phi$.
\end{rem}

\begin{deff}\label{def_goodred} Let $\phi=[F:G]$ 
be in normalized form with respect to $\p$. We say that $\phi$ has \textbf{good reduction at $\p$} if it satisfies one of the following equivalent conditions:
	\begin{enumerate}[noitemsep, label=\textnormal{(\alph{enumi})}]
		\item $\deg(\phi)=\deg(\widetilde{\phi})$;
		\item the equation $\widetilde{F}(X,Y)=\widetilde{G}(X,Y)=0$ has no solutions in $ \P^1(\overline{k(\p)})$;
		\item $\Res(\phi) \in R^*$;
		\item $\Res(\widetilde{F}, \widetilde{G}) \neq 0$.
	\end{enumerate}
Otherwise, we say that $\phi$ has \textbf{bad reduction at $\p$}.
\end{deff}

The equivalence of the conditions in Definition \ref{def_goodred} is a consequence of well known properties of the resultant, see \cite[Theorem 2.15]{Sil07}. If $\phi$ and $\psi$ have good reduction at $\p$, then the expected equalities $\widetilde{\phi}(\widetilde{P})= \widetilde{\phi(P)}$ and $\widetilde{\phi} \circ \widetilde{\psi}= \widetilde{\phi \circ \psi}$ hold. Moreover the composition $\phi \circ \psi$ has good reduction as well (see \cite[Theorem 2.18]{Sil07}). Therefore a periodic point $P$ reduces to a periodic point for $\widetilde{\phi}$, whose exact period divides that of $P$.

\begin{rem}\label{RGRed1}
		Let $\phi$ be a rational map defined over $K$ with good reduction at $\p$. Write $\phi=[F:G]$ and $P=[x: y]$ both in normalized form with respect to $\p$, then $\phi(P)= [F(x,y):G(x,y)]$ is again normalized with respect to $\p$. Indeed, by Definition \ref{def_goodred} (b) at least one of $F(x,y)$ and $G(x,y)$ is in $R^*$.	
\end{rem}

\begin{rem}\label{RGRed2} Consider the M\"obius transformation 
	\[\mu([X:Y])=[aX + b Y: c X + d Y] \] 
corresponding to the matrix 
\[\left[ \begin{matrix} a & b \\ c& d  \end{matrix} \right]\in\PGL_2(R).\] 
The resultant of $\mu$ is $ad - bc \in R^*$ and therefore $\mu$ has good reduction at $\p$. 
\end{rem}


\smallskip
\subsection{The $\p$-adic logarithmic distance}\label{subs:distance}
When working with rational maps defined over the projective line, the difference between two points will make no sense. In order to generalize the results of Section \ref{Spol} we will consider distances induced by the primes of good reduction.
\begin{deff}Let $P_1=[x_1:y_1]$ and $P_2=[x_2: y_2]$ be points in $\P^1(K)$. The \textbf{$\p$-adic logarithmic distance} on $\P^1(K)$ is defined as
	\begin{equation*} \delta_{\p}(P_1, P_2)= v_{\p}(x_1y_2 - x_2 y_1) - \min\{v_{\p}(x_1), v_{\p}(y_1)\} - \min \{v_{\p}(x_2), v_{\p}(y_2)\}.
	\end{equation*}
\end{deff}

\begin{rem}\label{Rdistance} The above definition is independent of the choice of homogeneous coordinates for the points. Furthermore, if $P_1$ and $P_2$ are in normalized form with respect to $\p$, their distance is simply
	\begin{equation}\label{eqd}
	 \delta_{\p}(P_1, P_2)=v_{\p}(x_1y_2 - x_2 y_1).
	\end{equation}
Notice that $\delta_{\p}(P_1, P_2) $ is always non negative; furthermore the distance between two points is strictly positive if and only if they reduce to the same point in $\P^1(k(\p))$.
\end{rem}

Two properties of the logarithmic $\p$-adic distance prove to be particularly helpful for the study of dynamical questions.

\begin{prop}\textup{\cite[Proposition 5.1]{MS95}}\label{P1}
For all $P, P', P'' \in \P^1(K)$
\begin{equation}\label{triang}
	\delta_{\p}(P,P') \geqslant \min\{\delta_{\p}(P, P''), \delta_{\p}(P'',P')\}\,.
\end{equation}
\end{prop}

\begin{prop}\textup{\cite[Proposition 5.2]{MS95}}\label{P2}
Assume that $\phi$ has good reduction at $\p$, then 
	\begin{equation*}
	\delta_{\p}(\phi(P), \phi(P')) \geqslant \delta_{\p} (P,P')
	\end{equation*}
	for every $P, P' \in \P^1(K)$.
\end{prop}

Proposition \ref{P2} says that a rational map with good reduction at $\p$ is everywhere non-contracting with respect to $\delta_{\p}$. In particular, when restricted to the orbit of a periodic point, the map $\phi$ preserves distances. 

\begin{prop}\textup{\cite[Proposition 6.1]{MS95}}\label{P3}
    Assume that $\phi$ has good reduction at $\p$ and let $P \in \P^1(K)$ be a periodic point with exact period $n$. Then
	\begin{enumerate}[noitemsep]
		\item $\delta_{\p}(\phi^i(P), \phi^j(P))= \delta_{\p}(\phi^{i+k}(P),\phi^{j+k}(P))$ for every $i, j, k \in\Z$;
		\item $\delta_{\p}(\phi^i(P), \phi^j(P))= \delta_{\p}(P,\phi(P))$ for all $i, j \in\Z$ satisfying $\gcd(j-i, n)=1$.
	\end{enumerate}
\end{prop}

The following property can be easily deduced from Proposition \ref{P3}.

\begin{lemma}\label{L0}
Assume that $\phi$ has good reduction at $\p$ and let $P \in \P^1(K)$ be a periodic point with exact period $n$. Then
	\[\delta_{\p}(\phi^i(P),\phi^j(P))= \delta_{\p}(P,\phi^d(P))\] 
	for all $i, j, d$ such that $\gcd(j-i,n)=d$. 
\end{lemma}

\begin{proof}
First $\delta_{\p}(\phi^i(P), \phi^j(P)) = \delta_{\p}(P, \phi^{j-i}(P))$ by Proposition \ref{P3} (i). Now set $\psi= \phi^d$ and apply Proposition \ref{P3} (ii) to the map $\psi$.
\end{proof}

The bound for cycle lengths obtained in this paper is in part a consequence of the following result. Here it is stated in a slightly different (and somewhat more general) form than in \cite{CP16}. We also give an alternative proof because there seems to be a small flaw in the original one. 

\begin{lemma}\label{L1new} \textup{\cite[Lemma 3.2]{CP16}}
    Let $\mathcal P$ be a set of points in $\P^1(K) $ such that the quantity $\delta_{\p} (P, Q) $ with $(P, Q) $ varying over the pairs of distinct points in $\mathcal P$ is constant. Then 
	 \[ \# \mathcal P \leqslant \#k(\p) + 1 \]
	\end{lemma}

\begin{proof} 
Let $\mathcal P = \{P_i : i = 0 , \dotsc , n-1 \}$; we first show that we may assume $P_0=[0:1]$. Indeed, for every $i=0, \dotsc , n-1$ write $P_i=[x_i: y_i]$ in normalized form with respect to $\p$. In particular there exist $a, b \in R$ such that $ax_0 + by_0 = 1$. Now set $Q_i= \mu (P_i)$, where $\mu $ is the M\"obius transformation associated with the matrix
\begin{equation*}
M= \left[\begin{matrix}
y_0 & -x_0 \\
a & b
\end{matrix}\right] \in \PGL_2(R).
\end{equation*} 	
Note that $ \mu([X:Y])=[y_0X -x_0 Y: a X + b Y] $ is written in normalized form and has good reduction at $\p$. By Remarks \ref{RGRed1} and \ref{RGRed2}, the points $Q_i=[y_0x_i -x_0 y_i: a x_i + b y_i]=[x_i':y_i']$ are written in normalized form with respect to $\p$, furthermore $Q_0=[0:1]$. Now take $i \neq j$, then
\begin{equation*}
\delta_{\p}(Q_i, Q_j)=
v_{\p}\left( \det \left[  \begin{matrix}
x_i' & x_j'\\
y_i' & y_j'
\end{matrix}\right]\right)=
v_{\p}\left( \det \left( M \left[  \begin{matrix}
x_i & x_j\\
y_i & y_j
\end{matrix}\right]\right)\right)=
v_{\p}\left(\det \left[  \begin{matrix}
x_i & x_j\\
y_i & y_j
\end{matrix}\right] \right)= \delta_{\p}(P_i, P_j),
\end{equation*}
and we may replace the original points with the $Q_i$'s.

Next assume $P_0=[0:1]$ and $\delta_{\p}(P_i, P_j)= \delta $ for all $i \neq j$. If $\delta = 0$ the thesis is true because the $P_i$'s reduce to distinct points in $\P^1(k(\p))$. If $\delta$ is strictly positive let $\pi \in R$ be a uniformizer for the valuation $v_{\p}$, i.e. an element such that $v_{\p}(\pi)=1$. By hypothesis $v_{\p} (x_i)=\delta_{\p}(P_0, P_i)= \delta$ for every $i\neq 0$ and we can write $P_i=[\pi^{\delta}a_i: y_i]$ with $v_{\p}(a_i)= v_{\p}(y_i)=0$. Next consider the points $P_i'=[a_i:y_i]$ for $i= 1 , \dotsc , n-1$ and observe that none of them reduces to $[0:1]$ modulo $\p$. On the other hand if $1 \leqslant i < j \leqslant n-1$, then
\begin{equation*}
\delta_{\p}(P_i', P_j')= v_{\p}(a_iy_j - a_j y_i)= v_{\p}(\pi^{-\delta}(x_i y_j - x_j y_i))= \delta_{\p}(P_i, P_j) - \delta = 0.
\end{equation*}
Therefore the points $P_i'$ with $i=1, \dotsc , n-1$ reduce to distinct points of $\P^1(k(\p))$ different from $[0:1]$, hence $n-1 \leqslant \# k(\p)$.
\end{proof}


\section{Preliminaries}\label{preliminaries}
In this paper we are interested in rational maps defined over a global function field $K$, hence the need for \textit{global definitions} that take into account the behavior of $\phi$ with respect to \textit{all} primes of $K$. We have already observed that locally there exists a choice of homogeneous coordinates that is preferable, i.e. normalized with respect to the prime $\p$. Now let $S$ be a finite set of primes of $K$ containing $\infty$.\footnote{$S$ will be the set of primes of bad reduction: we require it to be nonempty for technical reasons. When working with number fields, $S$ should also contain the archimedean places, but this is not a concern in our setting since all primes are non-archimedean.}

\begin{deff}
We say that $P = [x : y]$ is written in \textbf{normalized form outside $S$} if it is in normalized form with respect to every prime outside $S$. We will also say that $[x : y]$ are \textbf{$S$-coprime coordinates} for $P$.

\noindent Write a rational map $\phi$ as
\[ \phi([X : Y]) = [F(X, Y): G(X, Y)] \]
with $F, G$ homogeneous polynomials in $K[X, Y]$. We say that the pair $[F : G]$ is in \textbf{normalized form outside $S$} if it is in normalized form with respect to every prime outside $S$. 

\noindent A rational map $\phi$ has \textbf{good reduction outside $S$} if it has good reduction at all primes not in $S$. 
\end{deff}

Furthermore, we define
\[ \begin{array}{ll}
\mathcal O_S = \{ \alpha \in K \text{ such that } v_{\p} (\alpha) \geqslant 0 \text{ for every } \p \notin S\} & \text{the ring of $S$-integers};\\
\mathcal O_S^* = \{ \alpha \in K \text{ such that } v_{\p} (\alpha) = 0 \text{ for every } \p \notin S\} &\text{the group of $S$-units}
\end{array}\]
and briefly recall some of their properties. To begin with, $\mathcal O_S$ is a Dedekind domain and its maximal ideals correspond to the primes outside $S$ \cite[Theorem 14.5]{Rosen02}. Furthermore its field of fractions is $K$ \cite[Proposition 3.2.5]{Sti09}. The group of $S$-units has rank $\#S-1$ by Dirichlet's unit theorem \cite[Proposition 14.2]{Rosen02} and its torsion part is the group of units of the full constant field. We will also need the fact that the ideal class group of $\mathcal O_S$ is finite (again by \cite[Proposition 14.2]{Rosen02}).

\begin{rem}
The following rational maps have good reduction (and are written in normalized form) outside $S$:
\begin{enumerate}[label = (\alph{enumi})]
    \item M\"obius transformations belonging to $\PGL_2(\mathcal O_S)$;
    \item polynomials in $\mathcal O_S[X]$ with leading coefficient in $\mathcal O_S^*$. If $K$ is a global function field and $S=\{\infty\}$, then $\mathcal O_S = \mathcal O$ as defined in Section \ref{Spol} and these are exactly the polynomials discussed there.
    \end{enumerate}
\end{rem}

It is not difficult to see that $P = [x : y]$ is written in normalized form outside $S$ if and only if $x \mathcal O_S + y \mathcal O_S = \mathcal O_S$, thus explaining the use of the term \textit{$S$-coprime coordinates}. Therefore a point $P=[x:y]$ can be written in $S$-coprime coordinates if and only if the fractional ideal $x \mathcal O_S + y \mathcal O_S$ is principal.
It would be nice to be able to write each point of the projective line in $S$-coprime coordinates: then the $\p$-adic logarithmic distance would be expressed by \eqref{eqd} for every $\p \notin S$. The difficulty arises from the fact that in general $\mathcal O_S$ is not a principal ideal domain. To avoid this problem we can take coordinates in a larger field: this strategy was first devised by Canci in \cite{Canci07} and it has been applied in several other papers.

We will state the result and show how to apply it to the present situation, but first we need some notation. Let $L$ be a finite normal extension of $K$ and $\mathbb S$ the set of primes of $L$ lying above the primes in $S$: we shall work with $\mathbb{S}$-integers $\mathcal{O}_{\mathbb{S}}$ and $\mathbb{S}$-units $\mathcal{O}_{\mathbb{S}}^*$ in $L$.
For any subset $\mathcal S$ of $K$ we denote by $\sqrt{\mathcal S}$ the \textit{radical} of $\mathcal S$ in $L$, i.e. $\sqrt{\mathcal S} = \{ \alpha \in L \text{ such that } \alpha^n \in \mathcal S \text{ for some } n \in \mathbb N \}$. Then it is immediate that $\sqrt{\mathcal O_S} = \mathcal O_{\mathbb S} \cap \sqrt{K}$ and $ \sqrt{\mathcal O_S^*} = \mathcal O_{\mathbb S}^*\cap \sqrt{K^*}$. 

\begin{deff}
With notation as above, let $\mathcal P$ be a set of points in $\P^1(L)$. An \textbf{$S$-radical choice of coordinates for $\mathcal P$} is given by coordinates $[x_P : y_P]$ for $P\in \mathcal{P}$ such that:
\begin{enumerate}
    \item $x_P,y_P\in L$ and $P$ is written in normalized form outside $\mathbb S$;
    \item $x_P y_Q - x_Q y_P \in \sqrt{K^*}$ for every pair of distinct points $P,Q\in \mathcal{P}$.
\end{enumerate}
\end{deff}

\begin{lemma}\label{Lcoord}\textup{\cite[Lemma 4.4]{CTV19}}
	Let $K$ be a global function field and $S$ a finite set of primes containing $\infty$. Then there exist a finite normal extension $L/K$ and an $S$-radical choice of coordinates for the set $\P^1(K)$.
\end{lemma}

Lemma \ref{Lcoord} was originally stated and used on number fields, however its proof only relies on the finiteness of the ideal class group of $\mathcal O_S$, so it holds for global function fields as well.

For every prime $\p$ of $K$ we fix a prime $\mathfrak P$ in $L$ above $\p$ and, with a little abuse of notation, for any $\alpha\in \sqrt{K^*}$ we shall still write $v_\p(\alpha)$ to denote $v_{\mathfrak{P}}(\alpha)$.\footnote{The two valuations coincide over $K$ up to multiplication by the ramification index: we allow $v_{\p}$ to assume fractional values over $\sqrt{K}$ instead of taking the normalized valuation over $L$. Then it is possible to express the $\p$-adic distance as in \eqref{eq:dist}.}
Now fix an $S$-radical choice of coordinates for the projective line $\P^1(K)$. For every prime of $K$ the $\p$-adic logarithmic distance between two points is then expressed by 
\begin{equation}\label{eq:dist}
    \delta_{\p} (P, Q ) = v_{\p} (x_P y_Q - x_Q y_P).
\end{equation}

Next consider the following situation: $\phi$ is a rational map defined over $K$ with good reduction outside $S$ and $P_0 \in \P^1(K) $ a periodic point of minimal period $n$. Write its orbit as 
\begin{equation} \label{eq:cycle} P_0 \mapsto P_1 \mapsto \cdots \mapsto P_{n-1} \mapsto P_0 \end{equation}
with $P_i = [x_i : y_i]$ (in $S$-radical coordinates). Working as in the proof of Lemma \ref{L1new}, we consider the map $\mu$ associated with the matrix of determinant $1$
\[M = \left[\begin{matrix}
y_0 & -x_0 \\
a & b
\end{matrix}\right] \in \PGL_2(\mathcal O_{\mathbb S}).\] It is easy to check that for every pair of points $P, Q$ the quantity $x_P y_Q - x_Q y_P$ is $\mu$-invariant (again, see the proof of Lemma \ref{L1new} for details). Then, after replacing $P$ with $\mu(P)$ and $\phi$ with $\mu \circ \phi \circ \mu^{-1}$, we can assume that:
\begin{enumerate}
    \item $\phi$ is a rational map defined over $L$ with good reduction outside $\mathbb S$;
    \item $P_0 = [0 : 1]$ is a periodic point for $\phi$ of exact period $n$;
    \item there exists an $S$-radical choice of coordinates for the orbit of $P_0$.
\end{enumerate}

Next we are going to introduce some notation that will be used throughout Section \ref{Sproof}.
For every prime of good reduction, applying repeatedly the triangle inequality \eqref{triang} and Proposition \ref{P3} (i), we obtain
\begin{align*}
	\delta_{\p} (P_0, P_i) \geqslant \delta_{\p}(P_0, P_1) &\quad\text{for every } \p \notin \mathbb S,\text{ i.e.}\\
	v_{\p}(x_i) \geqslant v_{\p}(x_1) &\quad\text{for every } \p \notin \mathbb S.
\end{align*} 
Moreover $x_i = x_i \cdot 1 - 0 \cdot y_i \in \sqrt{K^*}$ for every $i$. Now define $A_i = x_i / x_1$, then $A_i$ is a nonzero element of $\mathcal O_{\mathbb S}\cap \sqrt{K^*}=\sqrt{\mathcal O_S}$. Similarly, it is not difficult to see that if $i$ divides $j$ modulo $n$, then $A_i$ divides $A_j$ in $\sqrt{\mathcal O_S}$. In particular if $i$ is relatively prime to $n$, then $A_i$ divides $A_1 = 1$, meaning that $A_i$ belongs to $\sqrt{\mathcal O_S^*}$. Then up to multiplication of both coordinates of $P_i$ by an element of $\sqrt{\mathcal O_S^*}$ we may assume $A_i = 1$ without affecting the $S$-radical choice of coordinates.
To sum it up:
\begin{enumerate}
    \item $P_0 = [0:1], \, P_i = [A_i x_1 : y_i]$ for every $i \neq 0$ is an $S$-radical choice of coordinates for the cycle of $P_0$;
    \item $A_i$ is nonzero and belongs to $\sqrt{\mathcal O_S}$ for every $i \neq 0$;
    \item $A_i = 1$ for every $i$ prime with $n$.
\end{enumerate}

We conclude this section with a formula which will be crucial in the proofs of the next section: take three indexes $(j, i, d) $ such that $\gcd(i-j,n)=d$, then by Lemma \ref{L0}
	\begin{align}
	\delta_{\p}(P_i, P_j)= \delta_{\p}(P_0,P_d) &\quad\text{for every $\p \notin \mathbb S$,} \nonumber \\
	v_{\p}(x_1(A_iy_j-A_jy_i))=v_{\p}(x_1 A_d) &\quad\text{for every $\p \notin \mathbb S$, i.e.}\nonumber \\
	A_i y_j - A_j y_i=A_d u &\quad\text{for some $u \in \sqrt{\mathcal{O}_S^*}.$}\label{eq3}
	\end{align}


\section{Cycles for rational maps}\label{Sproof}

In the remainder of this paper we will consider rational maps $\phi$ defined over a global function field $K$ with exactly one prime of bad reduction. Let $\F_q$ denote the full constant field of $K$ and $S$ the set of primes of bad reduction: in this case $\mathcal O_S^* = \F_q^*$. Let $L$ be the finite normal extension given by Lemma \ref{Lcoord}: $L$ can be explicitly constructed, provided that one knows the ideal class group of $\mathcal O_S$ (the reader is referred to the discussion before \cite[Lemma 4.4]{CTV19} for details), but in general we do not even know its degree over $K$. The group $\sqrt{\mathcal O_S^*}$ is by definition the radical of $\F_q^*$ in $L^*$, i.e. $\sqrt{\mathcal{O}_S^*}=\F^*$, where $\F$ is the full constant field of $L$.

The aim of this section is to prove Theorem \ref{T1}, which can be regarded as a generalization of equation \eqref{eqpol}. We will then apply Lemma \ref{L1new} to bound the length of cycles in terms of the characteristic of the field and the degree of the extension $K/\F_p(t)$. A further application of Theorem \ref{T1} can be found in Section \ref{Sdbound}, where we discuss bounds that depend uniquely on the degree of the map (with one caveat, as in the polynomial case).

\begin{theo}\label{T1}
	Let $\phi \colon \P^1\rightarrow \P^1$ be a rational map defined over a global function field $K$ with one prime of bad reduction and let $\mathcal C \subset \P^1(K)$ be a cycle of length $n >1$. Then for every fixed prime of good reduction $\p$ the quantity $\delta_{\p}(P, Q)$ with $(P, Q)$ varying over the pairs of distinct points in $\mathcal C$ is constant.
\end{theo}

\begin{rem}\label{RLenRedCyc}
Let $\p$ be a prime of good reduction for $\phi$, then a periodic point of minimal period $n$ reduces modulo $\p$ to a periodic point for $\widetilde{\phi}$ of exact period $m$ and $m$ divides $n$. A powerful theorem by Morton--Silverman and Zieve (see \cite[Theorem 2.21]{Sil07}) provides the possible values for $m$ (for rational maps of degree at least two). In the present situation ($\phi$ defined over a global function field and with one prime of bad reduction) Theorem \ref{T1} yields $m = n$ or $m = 1$. Indeed let $\p$ be a prime of good reduction, then for every two distinct points $P, Q$ belonging to an $n$-cycle, either $\delta_{\p} (P, Q) = 0 $ (that is, $\widetilde{P} \neq \widetilde{Q}$), or $\delta_{\p} (P, Q) > 0 $ (therefore $\widetilde{P} = \widetilde{Q}$). In the first case the length of the reduced cycle is still $n$, in the latter all points in the cycle reduce to a unique fixed point for $\widetilde{\phi}$.
\end{rem}

We take care of most of the technicalities needed for the proof of Theorem \ref{T1} in the following lemma.

\begin{lemma}\label{LT1}
Let $K$ be a global function field, fix a prime of $K$ at infinity and set $S = \{ \infty \}$. Take $L$ a finite normal extension of $K$ and $\mathbb S$ the set of primes of $L$ above the prime at infinity. Take $\phi$ and $P_0$ such that
\begin{enumerate}[label = \textnormal{(\alph{enumi})}]
    \item $\phi$ is a rational map defined over $L$ with good reduction outside $\mathbb S$;
    \item $P_0 = [0 : 1]$ is a periodic point for $\phi$ with orbit $\{P_0,\dots,P_{n-1}\}$;
    \item $P_0 = [0:1], \, P_i = [A_i x_1 : y_i]$ for $i\neq 0$ is an $S$-radical choice of coordinates (with the notation introduced in Section \ref{preliminaries}).
\end{enumerate}
Then $A_i$ belongs to $\F^*$ for every $i = 1 , \dotsc , n-1$.
\end{lemma}

\begin{proof}[Proof of Theorem \ref{T1}]
Fix the prime of bad reduction for $\phi$ at infinity and let $S = \{ \infty \}$.
Fix an $S$-radical choice of coordinates for the projective line $\P^1(K)$ and let $\mathcal C = \{ P_i : i=0,\dots,n-1 \}$. After performing the reduction step described in Section \ref{preliminaries}, we can apply Lemma \ref{LT1}, then $P_0 = [0:1]$ and $P_i = [A_i x_1 :y_i]$ with $A_i \in \F^*$ for every $i\neq 0$. Then for every prime of good reduction
\begin{equation} \label{eq:T1}
\delta_{\p} (P_0, P_i ) = v_{\p} (A_i x_1) = v_{\p} (x_1) = \delta_{\p} (P_0 , P_1) \quad \text{for every } i \neq 0.
\end{equation}
Now take any two distinct points in $\mathcal C$, let us say $P_i$ and $P_j$ with $i < j$, then for every good prime $\p$
\begin{align*}
     \delta_{\p} (P_i, P_j) &= \delta_{\p}(P_0, P_{j-i})
     &&\hspace{-8em}\text{by Proposition \ref{P3} (i)}
    \\
    &=\delta_{\p} (P_0 , P_1) &&\hspace{-8em}\text{by \eqref{eq:T1}.} 
    \qedhere
\end{align*}
\end{proof}

For the proof of Lemma \ref{LT1} we heavily use the construction carried on in Section \ref{preliminaries}. 
We will apply almost exclusively equation \eqref{eq3}, but the proofs will also rely on the fact that $\sqrt{\mathcal O_S^*} = \F^*$ is the group of units of a field, a circumstance that is specific to this setting.
A simple counterexample to Theorem \ref{T1} when this condition is not satisfied is given at the end of the section (Example \ref{Enfield}).

\begin{proof}[Proof of Lemma \ref{LT1}]
The thesis follows from the statements below:
\begin{enumerate}
    \item if there exists an index $i$ such that $A_i \notin \F^*$, then $n$ is even and $A_2 \notin \F^*$;
    \item if $n \equiv 0 \pmod 4$, then $A_i \in \F^*$ for every $i\neq 0$;
    \item if $n \equiv 2 \pmod 4$, then $A_i \in \F^*$ for every $i\neq 0$. 
\end{enumerate}
As remarked in Section \ref{preliminaries} we can assume $A_i=1$ for every $i$ prime with $n$.

We begin with the proof of (i): let $m$ be the smallest positive integer satisfying $A_m \notin \F^*$ and assume $m >2$. Since $A_i \in \F^*$ for every $i < m$, we can assume $A_i = 1$ for every $i < m$ after multiplication by an element of $\F^*$. Take $i=1, 2$ and set $d_i=\gcd(m-i, n)$, then $d_i < m$. By replacing $j$ with $m$ in \eqref{eq3} (with $A_i =A_{d_i} = 1$ by the previous remarks) we have 
\begin{equation*}
A_m y_i - y_m=u_i \quad \text{with $u_i \in \F^*$ for $i=1,2$.}
\end{equation*}
Note that $y_1 \neq y_2$ since $P_1=[x_1:y_1]$ and $P_2=[x_1:y_2]$ are distinct points. By subtracting the two previous equations we get \[0 \neq A_m (y_1 - y_2)= u_1-u_2.\]
Then $A_m (y_1 - y_2) \in \F^*$ and since $y_1$ and $y_2$ are both $\mathbb S$-integers, $A_m$ is in $\F^*$, a contradiction. Then $m = 2$, that is $A_2 \notin \F^*$ and $n$ is even because $A_i = 1$ for every $i$ prime with $n$.

Next we prove (ii). Assume that $n$ is divided by $4$ and suppose the thesis is false, in particular $A_2 \notin \F^*$ by (i). Let $n'=n/2$, then $n'$ is even and $A_2$ divides $A_{n'}$ in $\sqrt{\mathcal O_S} $. We will show that $A_{n'}$ belongs to $\F^*$, thus contradicting $A_2 \notin \F^*$.
Note that if $\gcd(j,n)=1$, then $\gcd(n'-j, n)=1$ as well because $n$ and $n'$ are divided by the same primes. 
Then equation \eqref{eq3} applied to $(j,n',1)$ yields 
\begin{equation*}
	A_{n'} y_j - y_{n'}=u_j \quad \text{with $u_j \in \F^*$.}
\end{equation*}
Next take an index $k \neq j$ such that $\gcd(k, n)=1$ as well (it is possible because $\varphi(n)\geq 2$), then again 
\begin{equation*}
A_{n'} y_k - y_{n'}=u_k \quad \text{with $u_k \in \F^*$}
\end{equation*}
and subtracting the two expressions (keep in mind that $y_j \neq y_k$ since $P_j \neq P_k$) we get
	\[0 \neq A_{n'} (y_j-y_k)=u_j - u_k.\]
Then $A_{n'} (y_j - y_k) \in \F^*$ and all the terms in the expression are $\mathbb S$-integers, therefore $A_{n'} \in \F^*$ as well.

The proof of (iii) is similar. Let $n= 2 n'$ with $n'$ an odd positive integer. The case $n = 2$ is trivial, therefore we set $n' >1$. First, by applying (i) to the map $\phi^2$ we may assume $A_i = A_2$ for every nonzero even index $i$. Next we prove that $A_i$ belongs to $\F^*$ for every odd index $i$ dividing $n$. Indeed, take $j,k$ distinct indexes such that $\gcd(n'-j, n)= \gcd(n'-k, n)=1$: this is possible because $\varphi(n) \geqslant 2$. Observe that $j$ and $k$ are even, hence we can write $P_j=[A_2x_1: y_j]$ and $P_k= [A_2 x_1: y_k]$. Apply \eqref{eq3} to $(j, n', 1) $ and $(k, n', 1)$ respectively and obtain
\begin{align*}
	 	A_{n'} y_j - A_2 y_{n'}= u_j \quad &\text{with $u_j \in \F^*$}\\
		A_{n'} y_k - A_2 y_{n'}= u_k \quad &\text{with $u_k \in \F^*$,}
\end{align*} 
then reasoning as in the previous step we get $A_{n'} \in \F^*$. Now, if $i$ is odd and divides $n$, then it divides $n'$: hence $A_i$ divides $A_{n'}$ in $\sqrt{\mathcal O_S}$, that is $A_i \in \F^*$.

If we prove that $A_2$ belongs to $\F^*$ we are done. Take $i=1, n-1$ (then $A_i=1$), and let $d_i= \gcd(2-i, n)$. Then $d_i$ is odd and divides $n$, therefore we can assume $A_{d_i} = 1$. Apply equation \eqref{eq3} to $(i, 2, d_i)$, then
	\begin{equation*}
	    A_2 y_i - y_2=u_i \quad \text{with $u_i \in \F^*$ for $i=1,n-1$.}
	\end{equation*}
	Subtracting the two expressions we get
	\[0 \neq A_2(y_1-y_{n-1})=u_1-u_{n-1}\]
	and finally $A_2 \in \F^*$.
\end{proof}


\begin{cor}\label{Cbound0}
Let $\mathcal C \subseteq \P^1(K)$ be a cycle for the map $\phi$. Then 
\begin{enumerate}
	\item $\# \mathcal C \leqslant \# k(\p) + 1$ for every prime of good reduction $\p$;
	\item if $[K:\F_p(t)]=D$, then $\# \mathcal{C} \leqslant p^D+1$.
\end{enumerate}
\end{cor}

\begin{proof}By Theorem \ref{T1} the first statement is immediate from Lemma \ref{L1new}. For the second just note that
there is at least one prime $\p$ in $K$ which is of good reduction for $\phi$ and lies over a prime of degree one of $\F_p(t)$. Since its inertia degree is at most $D$, its residue field $k(\mathfrak{p})$ has cardinality at most $p^D$. 
\end{proof}


We conclude this section with two examples.
First we go back to the issue of optimality already treated in Example \ref{ExOptCycLenght0}.

\begin{example}\label{ExOptCycLenght}
	For a fixed $q$, let $\F_q=\{0, 1, w_i : i=1, \dotsc , q-2\}$ and take a monic polynomial $\psi \in \F_q[X]$ of degree $2q-2 \geqslant q$ with a unique cycle of the form 
	\begin{equation*}
	0 \mapsto 1 \mapsto w_1 \mapsto \cdots \mapsto w_{q-2} \mapsto 0.
	\end{equation*}
	In Example \ref{ExOptCycLenght0} we proved that such a polynomial exists.
	Now define \[\phi(X)=\dfrac{ \psi(X)}{X^{2q-2}} \, ,\]
	then $0$ has the following periodic orbit of length $q + 1$:
	\begin{equation*}
	0 \mapsto \infty \mapsto 1 \mapsto w_1 \mapsto \cdots \mapsto w_{q-2} \mapsto 0. 
	\end{equation*}
\end{example}

It is perhaps natural to ask if Theorem \ref{T1} admits a corresponding result for number fields (in particular since we assume $\# S=1$ and $S$ should contain all archimedean places, we are only considering imaginary quadratic fields and $\Q$). More precisely: if $K$ is $\Q$ or an imaginary quadratic field and $\phi$ a rational map over $K$ with good reduction at every finite prime, is it true that the points of a cycle are all equidistant from each other with respect to the $\p$-adic logarithmic distance for every finite prime $\p$? The answer is no and we can give a fairly simple counterexample.
\begin{example}\label{Enfield}
	Let $K= \Q(i)$ and let $S$ contain the unique archimedean place of $K$. Then $\mathcal{O}_S= \Z[i]$ and $\mathcal{O}^*_S=\{\pm 1, \pm i \}$. Consider the polynomial $\phi(X)= i X$, then $\phi$ has good reduction at every finite prime and a cycle of length $4$
	\begin{equation*}
	1 \mapsto i \mapsto -1 \mapsto -i \mapsto 1.
	\end{equation*}
	Since $4>2+1$ we expect Theorem \ref{T1} to fail for the prime lying over $2$, that is $\p= (1 + i)$. Indeed $\delta_{\p}(1, i)= v_{\p}(1-i)= 1$, while $\delta_{\p}(1, -1)= v_{\p}(2)= 2$.
\end{example}

\section{Finite orbits}\label{Sfinite}

We now deal with bounds of the type of Corollary \ref{Cbound0} for the cardinality of finite orbits. 
We begin with two lemmas about rational maps defined over valued fields. The first one is a result that relates the dynamical behavior of a map with good reduction to metric properties, very much like Propositions \ref{P2} and \ref{P3}. 

\begin{lemma}\label{Lpreper}\textup{\cite[Lemma 4.1]{CP16}}
	Let $K$ be a valued field and $\phi \colon \P^1 \rightarrow \P^1$ a rational map defined over $K$ with good reduction at $\p$. Let $\mathcal A \subseteq \P^1(K)$ be a finite orbit containing a fixed point $P_0$ and write 
	\begin{equation}\label{eq:finite}
	P_{-m+1} \mapsto P_{-m+2} \mapsto \cdots \mapsto P_{-1} \mapsto P_0\mapsto P_0.
	\end{equation}
    Then for every $a, b$ such that $1 \leqslant a < b \leqslant m-1$
	\begin{equation*}
	\delta_{\p}(P_{-b}, P_{-a}) = \delta_{\p}(P_{-b} , P_0) \leqslant \delta_{\p}(P_{-a}, P_0).
	\end{equation*}
\end{lemma}

A finite orbit $\mathcal A$ for a map $\phi$ consists of a periodic part and a \textit{tail} part (i.e. all the points in $\mathcal A$ that are not periodic) denoted by $\Tail (\mathcal A)$. Assume the periodic part is a cycle of length $n$ and consider the behavior of the map $\phi^n$ over $\mathcal A$. Then all periodic points in the cycle become fixed points with respect to $\phi^n$ and for every tail point $P$ there exists a unique periodic point $Q \in \mathcal A$ satisfying $\phi^{nk}(P) = Q$ for some positive integer $k$. In other words, we may say that $Q$ is the unique periodic point in the forward orbit of $P$ with respect to the map $\phi^n$.

\begin{lemma}\label{LTroncoso}\textup{\cite[Corollary 2.23]{Tro17} }
	Let $K$ be a valued field and $\phi \colon \P^1 \rightarrow \P^1$ a rational map defined over $K$ with good reduction at $\p$. Let $P$ be a tail point for $\phi$ and $n$ the length of the periodic part of its orbit. If $Q$ is any periodic point satisfying $\widetilde{Q} = \widetilde{P}$, then $Q$ is the periodic point in the forward orbit of $P$ with respect to $\phi^n$.
\end{lemma}


The next result is of a different nature: in this case we consider points with coordinates in a global field $K$. Let $S$ be a finite set of primes of $K$; by putting together local information about $\p$-adic distances for all primes outside $S$ we establish a connection between points with certain properties and solutions of unit equations. 
The \textit{$n$-points lemmas} are central tools in the work of Canci, Troncoso and Vishkautsan over number fields (see \cite{CTV19} and \cite{CV19}) where they are applied in combination with bounds for the number of solutions of the resulting unit equations. The three points lemma we present here is a qualitative version of \cite[Corollary 4.2]{CTV19}; its proof can be adapted from that of \cite[Lemma 4.1]{CTV19} but we include it here for the reader's convenience.

\begin{lemma}[Three points lemma]\label{L3pts}
Let $K$ be a global function field and $S$ a finite set of primes of $K$ containing $\infty$. Fix an $S$-radical choice of coordinates for $\P^1(K)$. Let $Q_1, Q_2, Q_3$ be distinct points in $\P^1(K)$, with $Q_i = [x_i : y_i]$. Then there exists an injection from the set 
\[\mathcal{P} = \{P \in \P^1(K) \text{ such that } \delta_{\p} (P, Q_1) = \delta_{\p} (P, Q_2) = \delta_{\p} (P, Q_3) \text{ for all } \p \notin S \}\]
to the set of solutions $(u, w) \in \sqrt{\mathcal O_S^*} \times \sqrt{\mathcal O_S^*} $ of the unit equation 
\[ \dfrac{x_3 y_2 - x_2 y_3}{x_1 y_2 - x_2 y_1} \, u \,+ \, \dfrac{x_1 y_3 - x_3 y_1}{x_1 y_2 - x_2 y_1} \, w = 1 .\]
\end{lemma}

\begin{proof}
Take a generic point of the projective line $P\in \P^1(K)$ and write $P = [x : y]$, then the following conditions are equivalent:
\begin{align}
\delta_{\p} (P, Q_1) = \delta_{\p} (P, Q_3) & \quad \text{for all $\p \notin S$} \nonumber \\ 
v_{\p} (x_1 y - y_1 x)= v_{\p} (x_3 y - y_3 x) & \quad \text{for all $\p \notin \mathbb S$} \nonumber \\ 
x_1 y - y_1 x = u(x_3 y - y_3 x) & \quad \text{for some $u \in \sqrt{\mathcal O_S^*}$.} \label{eq:u}
\end{align}
Similarly, by taking $Q_2$ in place of $Q_1$:
\begin{equation} \label{eq:w}
 x_2 y - y_2 x = w(x_3 y - y_3 x) \quad \text{for some } w \in \sqrt{\mathcal O_S^*}.
\end{equation}
Thus $P\in \mathcal{P}$ if and only if there exists $(u,w) \in \sqrt{\mathcal{O}_S^*} \times \sqrt{\mathcal{O}_S^*}$ such that the linear system
	\begin{equation}\label{eq:L3pts1}
	\begin{cases}
		(u y_3 - y_1) X	+ (x_1 - u x_3) Y = 0\\
		(w y_3 - y_2) X + (x_2 - w x_3) Y = 0
	\end{cases}	
	\end{equation}
    has the nontrivial solution $(x, y)$. Then the determinant vanishes, that is
\begin{equation}\label{eq:L3pts2}
\dfrac{x_3 y_2 - x_2 y_3}{x_1 y_2 - x_2 y_1} \, u \,+ \, \dfrac{x_1 y_3 - x_3 y_1}{x_1 y_2 - x_2 y_1} \, w = 1\,.
\end{equation}
Conversely, fix $(u, w) \in \sqrt{\mathcal O_S^*} \times \sqrt{\mathcal O_S^*} $ a solution of \eqref{eq:L3pts2}, then the first equation in \eqref{eq:L3pts1} is a non-degenerate linear form (otherwise $Q_1= Q_3$) and therefore it defines a unique point in $\P^1(L)$, but not necessarily in $\P^1(K)$. What really matters is that the map $P \mapsto (u, w) $ admits a left inverse, i.e. it is injective.
\end{proof}


From now on we go back to the case of a rational map $\phi\colon\P^1\rightarrow\P^1$ defined over a global function field $K$ and with good reduction outside the set $S=\{\infty\}$. Recall that in these hypotheses $\sqrt{\mathcal{O}_S^*}=\F^*$ for some finite field $\F$. The two following lemmas describe the behavior of tail points with respect to the $\p$-adic distance and reduction modulo $\p$ for all primes of good reduction.
Both proofs rely on Lemma \ref{L3pts}: we shall produce unit equations with no solutions to prove that certain points cannot exist.

\begin{lemma}\label{Lfixed}
Let $\mathcal A \subseteq \P^1(K)$ be a finite orbit for $\phi$ with a fixed point as in \eqref{eq:finite}. Then
\begin{enumerate}
\item for every prime of good reduction $\p$ the quantity $\delta_{\p}(P, Q) $ with $(P, Q)$ varying over the pairs of distinct points in $\mathcal A - \{ P_{-m + 1} \}$ is constant;
\item $\# \mathcal A \leqslant \# k(\p) + 2$.
\end{enumerate}
\end{lemma}
\begin{proof}
The statement in (ii) readily follows from (i) applying Lemma \ref{L1new}. By Lemma \ref{Lpreper}, to prove (i) it suffices to show that for every prime of good reduction
\[ \delta_{\p}(P_{-i}, P_0)= \delta_{\p}(P_{-1}, P_0) \quad \text{for } i=2, \dotsc , m-2. \]
Let $a$ be the smallest positive integer satisfying $\delta_{\q} (P_{-a}, P_0) \neq \delta_{\q}(P_{-1}, P_0)$ for some good prime $\q$ (then of course $a > 1$). The thesis is satisfied if either $a$ does not exists or $a$ equals $m - 1$. Assume the contrary and put $Q_1 = P_{-1}, Q_2 = P_{-a}$ and $Q_3 = P_0$. Now take $P = P_{-m + 1}$, then by Lemma \ref{Lpreper} for every prime of good reduction $\p$ 
\[ \delta_{\p} (P, Q_1) = \delta_{\p} (P, Q_2) = \delta_{\p} (P, Q_3) \]
and we can apply Lemma \ref{L3pts}. Fix an $S$-radical choice of coordinates for $\P^1(K)$ (see Section \ref{preliminaries}) and let $Q_i = [x_i : y_i] $ for $i = 1, 2, 3$. Then the point $P$ corresponds to a solution $(u, w) \in \F^* \times \F^*$ of the equation $A u + B w= 1 $ with nonzero coefficients defined by 
 \begin{equation}\label{eq:coeff}
        A=\dfrac{x_3 y_2 - x_2 y_3}{x_1 y_2 - x_2 y_1} \quad \text{and} \quad B= \dfrac{x_1 y_3 - x_3 y_1}{x_1 y_2 - x_2 y_1}.
\end{equation}
 Now, $A$ and $B$ belong to $\sqrt{K^*}$ and for every good prime $\p$
    \begin{equation}\label{eq:val}
       \begin{aligned}
        v_{\p} (A) &= \delta_{\p} (Q_2, Q_3) - \delta_{\p} (Q_1, Q_2),\\
        v_{\p} (B) &= \delta_{\p} (Q_1, Q_3) - \delta_{\p} (Q_1, Q_2).
    \end{aligned} 
    \end{equation}
By Lemma \ref{Lpreper} for every prime of good reduction $\p$
\[\delta_{\p} (Q_2, Q_3) =\delta_{\p}(P_{-a}, P_0) = \delta_{\p}(P_{-a}, P_{-1})= \delta_{\p}(Q_2, Q_1)\,,\]
thus $A \in \F^*$. Furthermore, by the assumption on $a$, there exists a prime of good reduction $\q$ such that
\[ \delta_{\q} (Q_1, Q_2) = \delta_{\q} (P_{-a}, P_{-1}) = \delta_{\q} (P_{-a}, P_0) \neq \delta_{\q} (P_{-1}, P_0) = \delta_{\q} (Q_1, Q_3) \,.\]
Then $v_{\q} (B) \neq 0 $ and in particular $B \notin \F^*$. On the other hand, $B w = 1 - A u $ is not zero and belongs to $\F$, a contradiction.
\end{proof}

\begin{lemma}\label{Lfinite}
Let $\mathcal A \subseteq \P^1(K)$ be a finite orbit with periodic part of period $n \geqslant 4$. Then distinct points in the tail of $\mathcal A$ reduce to distinct points for every prime of good reduction.
\end{lemma}
\begin{proof}
We need to apply Lemma \ref{L3pts}, so we fix an $S$-radical choice of coordinates for $\P^1(K)$. We distinguish two cases depending on the behavior of the periodic part of the orbit under reduction modulo $\p$. Fix a prime of good reduction and let $m$ denote the length of the reduced cycle for the map $\widetilde{\phi}$. Recall that either $m = n$ or $m = 1$ (see Remark \ref{RLenRedCyc}), then one of the following is true:
\begin{enumerate}[label=\arabic{enumi}.] 
	\item there exists a good prime such that $m = 1$;
	\item $m = n $ for all primes of good reduction.
\end{enumerate}
\begin{description}
    \item[Case 1] Let $\q$ be a prime of good reduction for which $m = 1$. We prove that in this case $\Tail(\mathcal A)$ consists of at most one point, so that the thesis is trivially satisfied. Assume that $\Tail(\mathcal A)$ is nonempty and denote by $Q_1 \in \Tail(\mathcal{A})$ the point such that $\phi(Q_1) $ is periodic. Put $Q_2 = \phi(Q_1)$ and $Q_3 = \phi^2(Q_1)$: both  are periodic points.
    Assume that there exists a second point in $\Tail(\mathcal A)$, then we can take $P$ as the unique point in $\mathcal A$ satisfying $\phi(P) = Q_1$.
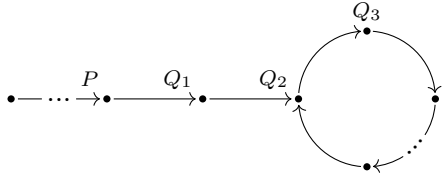
\begin{center}
\begin{tikzpicture}[scale=0.9, font=\scriptsize\bf\sffamily]
]
\foreach \x in {0, 1, 2, 3}
    \fill (90*\x: 1) circle (1.5pt);
\foreach \x in {0, 1, 2} 
    \draw[<-] (90*\x +5:1) arc (90*\x +5: 90*(\x+ 1)-5 : 1);
\draw[<-](90*3 +5:1) arc (90*3 +5: 90*3 + 32 : 1);
\draw[decorate, decoration={text along path, text={...}, text align={align=center}}] (90*3 : 1.02) arc (90*3 : 90*4 : 1.02);
\draw (90*3 + 58 : 1) arc (90*3 + 58 : 90*4 -5 : 1);
\foreach \x in {-1, -2, -3}
    \fill (180 : 1) ++ (\x*1.4, 0) circle (1.5pt); 
\foreach \x in {-1, -2}
	\draw[->] (180 : 1) ++ (\x*1.4 + 0.1, 0) -- ++ (1.2, 0); 
\draw[decorate, decoration={text along path, text={...}, text align={align=center}}] (-2*1.4- 1, 0.02) -- ++ (-1.4,0);
\draw (180 : 1) ++ (-3*1.4 + 0.1, 0) -- ++ (1.4*0.25, 0);
\draw[<-] (180 : 1) ++ (-2*1.4 -0.1, 0) -- ++ (-1.4*0.25, 0);
\node at (180 : 1) [above left]{$Q_2$};
\node at  (-1.4 - 1, 0) [above left]{$Q_1$};
\node at (-1.4*2 - 1, 0) [above left]{$P_{}$};
\node at (90 : 1) [above]{$Q_3$};
\end{tikzpicture} 
\end{center}
We want to prove that for every prime of good reduction $\p$
    \begin{equation}\label{eq:3pts1}
        \delta_{\p}(P, Q_i) =0 \quad \text{for } i = 1, 2, 3.
    \end{equation}
   This is equivalent to showing that for every good prime the reduced points $\widetilde{P}$ and $\widetilde{Q}_i$ are not the same (see Remark \ref{Rdistance}).
    For $i = 1,3$ it suffices to apply Lemma \ref{LTroncoso} since $P$ is a tail point, $Q_2$ and $Q_3$ are periodic and $\phi^n(P) \neq Q_2, Q_3$ (because $n\geqslant 4$). For $i=1$ assume $\widetilde{P} = \widetilde{Q}_1$ for some good prime, then by good reduction
    \begin{equation*}
        \widetilde{P} = \widetilde{\phi(P)} = \widetilde{\phi} (\widetilde{P}),
    \end{equation*}
    meaning that $\widetilde{P}$ is a fixed point for $\widetilde{\phi}$. But then it would also be true that $\widetilde{Q}_1 = \widetilde{Q}_2$, a contradiction to Lemma \ref{LTroncoso}. 
    
    Since \eqref{eq:3pts1} is true we can apply Lemma \ref{L3pts}: let $Q_i = [x_i : y_i]$, then the point $P$ corresponds to a solution $(u, w) \in \F^* \times \F^*$ of the equation $A u + B w = 1$, with nonzero coefficients defined by \eqref{eq:coeff}.
    Now, $A$ and $B$ belong to $\sqrt{K^*}$ and for every good prime the $\p$-adic valuations of the coefficients are expressed in terms of distances between the points $Q_i$ by \eqref{eq:val}.
    Since $\phi^n(Q_1)\neq Q_2, Q_3$, Lemma \ref{LTroncoso} yields $\delta_{\p}(Q_1, Q_2) = \delta_{\p} (Q_1, Q_3) = 0$ for every $\p\neq \infty$ and thus $B \in \F^*$. On the other hand $\delta_{\q}(Q_2, Q_3) > 0 $ by the assumption on the length of the reduced cycle modulo $\q$, hence $A \notin \F^*$. But $A u = 1 - B w $ is not zero and belongs to $\F$, a contradiction.

    \item[Case 2] Assume $m = n $ for every prime of good reduction. If $\Tail(\mathcal A)$ contains at most one point, then there is nothing to prove. If not, take $Q_1 \in \Tail (\mathcal{A})$ such that $\phi(Q_1)$ is periodic. Suppose there exist a prime of good reduction $\q$ and two distinct points in $\Tail (\mathcal A)$ that reduce to the same point modulo $\q$. We claim that in this case $Q_1$ and the periodic point $\phi^n(Q_1)$ reduce to the same point modulo $\q$. Indeed, let $P_1, P_2 $ be distinct tail points satisfying $\widetilde{P}_1 = \widetilde{P}_2$ modulo $\mathfrak{q}$: without loss of generality we can assume $P_2 = \phi^a(P_1)$ for some $a > 0$. Then $\widetilde{P}_1$ is periodic for the map $\widetilde{\phi}$, but its orbit contains a cycle of length $n$ (the reduced cycle for $\widetilde{\phi}$) by the assumption $m = n$, hence $\widetilde{P}_1= \widetilde{\phi}^n (\widetilde{P}_1) $. By applying a suitable iterate of $\widetilde{\phi}$ on both sides we see that
    \begin{equation}\label{eq:Q_1}
    \widetilde{Q}_1 = \widetilde{\phi}^n(\widetilde{Q}_1) = \widetilde{\phi^n(Q_1)}. 
    \end{equation}
    
Put $Q_2 = \phi(Q_1)$ and $Q_3 = \phi^n(Q_1)$: both are periodic points. Next take $P$ a periodic point in $\mathcal A$ distinct from $Q_2$ and $Q_3$. 
    
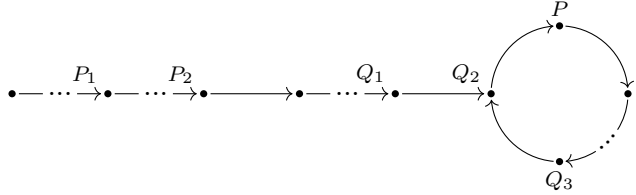
\begin{center}
\begin{tikzpicture}[scale = 0.9, font=\scriptsize\bf\sffamily]

\foreach \x in {0, 1, 2, 3}
    \fill (90*\x: 1) circle (1.5pt);
\foreach \x in {0, 1, 2}
    \draw[<-] (90*\x +5:1) arc (90*\x +5: 90*(\x+ 1)-5 : 1);
\draw[<-](90*3 +5:1) arc (90*3 +5: 90*3 + 32 : 1);
\draw[decorate, decoration={text along path, text={...}, text align={align=center}}] (90*3 : 1.02) arc (90*3 : 90*4 : 1.02);
\draw (90*3 + 58 : 1) arc (90*3 + 58 : 90*4 -5 : 1);

\foreach \x in {-1, -2, ..., -5}
    \fill (180 : 1) ++ (\x*1.4, 0) circle (1.5pt);

\foreach \x in {-1, -3}
 \draw[->] (180 : 1) ++ (\x*1.4 + 0.1, 0) -- ++ (1.2, 0); 
 
\foreach \x in {-2, -4, -5}{
    \draw[decorate, decoration={text along path, text={...}, text align={align=center}}] (180 :1) ++ (\x*1.4, -0.02) -- ++ (1.4, 0);
    \draw (180 : 1) ++ (\x*1.4 + 0.1, 0) -- ++ (1.4*0.25, 0);
    \draw[<-] (180 : 1) ++ (\x*1.4 + 1.3, 0) -- ++ (-1.4*0.25, 0);
    }
 
\node at (180 : 1) [above left]{$Q_2$};
\node at (-1.4 - 1, 0) [above left]{$Q_1$};
\node at (-1.4*3 - 1, 0) [above left]{$P_2$};
\node at (-1.4*4 - 1, 0) [above left]{$P_1$};
\node at (90 : 1) [above]{$P$}; 
\node at (270 : 1) [below]{$Q_3$};
\end{tikzpicture}
\end{center}
    We show that for every prime of good reduction $\p$ equation \eqref{eq:3pts1} holds.
    This is true for $i = 2,3$ because $P, Q_2$ and $Q_3$ are distinct periodic points and therefore reduce to distinct points modulo $\p$ for every prime of good reduction (by the hypothesis $m=n$). On the other hand $\widetilde{P} \neq \widetilde{Q}_1$ for every good prime by Lemma \ref{LTroncoso}. Then we can apply Lemma \ref{L3pts}: the point $P$ corresponds to a solution $(u, w) \in \F^* \times \F^*$ of the equation $A u + B w = 1$, with coefficients defined by \eqref{eq:coeff}.
   Similarly to the previous case we show that $A$ belongs to $\F^*$ while $B$ does not, which leads to a contradiction. Again, the $\p$-adic valuations of the coefficients are expressed in terms of distances by \eqref{eq:val}.
    For every prime of good reduction $\delta_{\p} (Q_2, Q_3 ) = 0$ because $Q_2$ and $Q_3$ reduce to distinct periodic points. Furthermore $\delta_{\p}(Q_1, Q_2) = 0 $ by Lemma \ref{LTroncoso}. Thereby $v_{\p} (A) = 0 $ for every good $\p$, that is $A \in \F^*$. On the other hand $\delta_{\q} (Q_1, Q_3) > 0 $ by equation \eqref{eq:Q_1}, therefore $v_{\q} (B) >0$ and in particular $B\notin \F^*$.
    \qedhere
\end{description}
\end{proof}

We now have all the elements needed to prove the bound for the cardinality of finite orbits.

\begin{theo}\label{TOrbits}
Let $\phi \colon \P^1 \rightarrow \P^1$ be a rational map defined over a global function field with one prime of bad reduction. Let $\mathcal A \subseteq \P^1(K)$ be a finite orbit for $\phi$, then 
\begin{enumerate}
	\item $\# \mathcal A \leqslant 3\cdot \# k(\p) + 6$ for every prime of good reduction $\p$;
	\item if $[K:\F_p(t)]=D$, then $\# \mathcal{A}\leqslant 3p^D+6$.
\end{enumerate}
\end{theo}
\begin{proof}
    Let $n$ denote the length of the cycle that is the periodic part of the orbit $\mathcal A$. If $n \geqslant 4$ we can apply Lemma \ref{Lfinite}: then there are at most $\# k(\p) + 1$ tail points in $\mathcal A$ (because their reductions modulo $\p$ are distinct). Moreover, by Corollary \ref{Cbound0} the number of periodic points in $\mathcal A$ is at most $\#k(\p) + 1$ as well, so that $\# \mathcal A \leqslant 2\cdot\# k(\p) + 2$.
    
    If $n \leqslant 3$, we consider the map $\phi^n$ instead. This gives a partition of $\mathcal A$ in $n$ orbits (with respect to $\phi^n$) with a fixed point as in \eqref{eq:finite} and we can apply Lemma \ref{Lfixed} (ii) to each one of them, so that $\# \mathcal A \leqslant 3\cdot \# k(\p) + 6$.

The second statement follows as in the proof of Corollary \ref{Cbound0}: just note that there exists a good prime $\p $ such that $\# k(\p) \leqslant p^D$.
\end{proof}

\section{Alternative bounds}\label{Sdbound}

The first result of this section is a generalization of Proposition \ref{Ppol} to rational maps. As a consequence we will prove (in the same hypotheses of Theorem \ref{Tbound}) that cycle lengths can be bounded solely in terms of the degree of the rational map $\phi$, unless $\phi$ is defined over a finite field. This condition on the map $\phi$ is to be expected, as the following example shows.

\begin{example}
    Consider $\phi\colon z \mapsto z^2$ over some finite field $\F_q$. Of course $\phi$ may also be regarded as a map over $\F_q(t)$ with good reduction everywhere. Assume $q$ is even and take a generator $w$ of $\F_q^*$: obviously $w$ has exact period $\log_2 q$ for $\phi$. Therefore, as $q$ varies, $w$ has minimal period greater than any given constant.
\end{example}

\begin{prop}\label{Prational}
Let $\phi \colon \P^1 \rightarrow \P^1$ be a rational map of degree $d$ defined over a global function field $K$ and fix a prime of $K$ at infinity. If $\mathcal P$ and $\mathcal P'$ are subsets of $\P^1(K) $ satisfying:
\begin{enumerate}[label = \textnormal{(\alph{enumi})}]
	\item for every prime $\p \neq \infty$ the quantity $\delta_{\p}(P, Q)$ with $(P, Q) $ varying over the pairs of distinct points in $\mathcal P$ is constant;
	\item $\mathcal P' \cup \phi( \mathcal P') \subseteq \mathcal P$.
\end{enumerate}
Then either $\# \mathcal P' \leqslant 2d$, or after a change of coordinates $\phi$ is defined over a finite field.
\end{prop}

\begin{proof}
Let $\overline K$ denote a fixed algebraic closure of $K$. The structure of the proof is as follows. 
\begin{enumerate}
	\item There exist a finite field $\F$ and a linear fractional transformation in $\PGL_2(\overline K)$ which induces a map $\mu\colon \mathcal{P} \rightarrow \F$.
	\item If $z_0, z_1, \dotsc , z_{2d} \in \F$ are $2d + 1$ distinct points such that $\mu\circ\phi\circ \mu^{-1}(z_i ) = w_i \in \F$, then there exists a rational map $\psi$ defined over $\F$ with $\deg(\psi) \leq d$ satisfying 
	\begin{gather*}
		\psi (z_i) = \mu \circ \phi \circ \mu^{-1} (z_i )= w_i \quad \text{for every } i = 0 ,\dotsc ,2d.
	\end{gather*}
	\item The maps $\psi $ and $\mu\circ\phi\circ\mu^{-1}$ are the same.
\end{enumerate}

We begin with the proof of (iii): assume that $\psi$ is a rational map of degree less than or equal to $d$ such that $\psi(z_i) = \mu \circ \phi \circ \mu^{-1}(z_i) $ for all $i$.
Taking a common denominator we can write $\mu \circ \phi \circ \mu^{-1}(X)- \psi(X) = f(X)/g(X) $ as the quotient of two relatively prime polynomials with $\deg(f) \leqslant 2d$. Thus $f(z_i) = 0$ for every $i=0,\dots,2d$ and $f$ is the zero polynomial, that is $\psi =\mu \circ \phi \circ \mu^{-1}$.

We now prove (i): fix an $S$-radical choice of coordinates for $\P^1(K)$ and put $\mathcal P = \{ {P_i=[x_i:y_i]} : i = 0, \dotsc , n-1 \}$. Working as in Section \ref{preliminaries} we can assume $P_0 = [0 : 1]$: recall that the change of coordinates does not affect $\p$-adic distances. By (a) for every nonzero index $i$
\begin{align*}
	\delta_{\p}(P_0, P_i) = \delta_{\p}(P_0, P_1) & \quad \text{for every } \p \notin S \text{ yields}\\
	v_{\p}(x_i) = v_{\p}(x_1) & \quad \text{for every } \p \notin \mathbb S \text{, i.e.}\\
	x_i =u_i x_1 & \quad \text{for some } u_i \in \F^*.
\end{align*}
Then up to multiplication of both coordinates by an element of $\F^*$ we may assume $P_i = [x_1 : y_i]$ for every nonzero index $i$. Again by (a) for $i > 1$
\begin{align*}
\delta_{\p} (P_1, P_i) = \delta_{\p}(P_0, P_1) & \quad \text{for every } \p \notin S \text{ yields}\\
v_{\p} (x_1 (y_i - y_1) ) = v_{\p} (x_1) & \quad \text{for every } \p \notin \mathbb S \text{, i.e.}\\
y_i = y_1 + u_i & \quad \text{for some } u_i \in \F^*.
\end{align*}
Then $\mathcal P$ contains $P_0 = [0:1]$ and $P_i = [x_1: y_1 + u_i]$ with $u_i \in \F$ for $i = 1, \dotsc , n-1$ (of course $u_1 = 0$). Up to considering a finite extension, we can assume that $\F$ contains an element $u$ such that $u + u_i \neq 0 $ for all $i$. Consider the M\"obius transformation $\mu$ associated with
\[M= \left[ \begin{matrix} 1 & 0 \\ u - y_1 & x_1
\end{matrix}\right] \in \PGL_2(\overline K),\]
that is $\mu(P) = \frac{x_P}{(u - y_1)x_P + x_1 y_P}$. Then $\mu(P_0)=0$ and $\mu(P_i)=(u+u_i)^{-1}$ for every $i>0$, therefore $\mu(\mathcal P ) \subseteq \F$.

For the proof of (ii) we abuse notation a bit by replacing $\phi$ with $\mu\circ\phi\circ \mu^{-1}$, so that $z_0, \dotsc, z_{2d}$ are distinct elements of a finite field $\F$ and $\phi(z_i) = w_i \in \F$ for all $i$. Write the generic rational map defined over $\overline K$ of degree at most $d$ as
\begin{equation*}
\psi (X) = \dfrac{a_0 + a_1 X + \cdots + a_d X^d}{b_0 + b_1 X+ \cdots + b_d X^d}
\end{equation*}
and set $a = (a_0,\dots,a_d)$ and $ b = (b_0,\dots,b_d)$ vectors of $\overline K^{d+1}$. We can identify the rational map $\psi$ with the vector $(a,b) $ of length $2d + 2$. Of course for every $\lambda \in \overline K^*$ the vectors $(a, b) $ and $ (\lambda a, \lambda b) $ define the same rational function. This gives a map from the set of rational functions of degree at most $d$ to the set of lines in $\overline K^{2d + 2}$:
\begin{equation*}
\psi (X) = \dfrac{a_0 + a_1 X + \cdots + a_d X^d}{b_0 + b_1 X+ \cdots + b_d X^d}\longmapsto [a:b] \in \P(\overline K^{2d+2})\,.
\end{equation*}
By analogy with the proof of Proposition \ref{Ppol}, we want to write the condition
\begin{equation}\label{eq:psi1}
\psi(z_i ) = w_i \quad \text{for } i = 0, \dotsc, 2d
\end{equation}
as a linear system where the coefficients of $\psi $ are treated as indeterminates.
First, notice that equation \eqref{eq:psi1} implies 
\begin{equation}\label{eq:psi2}
\sum_{j = 0}^{d} a_j z_i^j = w_i \sum_{j=0}^{d} b_j z_i^j \quad \text{for } i = 0, \dotsc, 2d.
\end{equation}
Now define the $(2d + 1) \times (d + 1) $ matrix with coefficients in $\F$
\begin{equation*}
Z= \left[ \begin{array}{cccc} 1 & z_0 & \cdots & z_0^d \\1 & z_1 & \cdots & z_1^d \\ \vdots &\vdots & & \vdots\\ 1 & z_{2d} & \cdots & z_{2d}^d
\end{array}\right]
\end{equation*}
and let $\{ e_1, \dotsc , e_{2d + 1 }\} $ denote the canonical basis of $\overline K^{2d +1}$, then \eqref{eq:psi2} can be rewritten as 
\begin{equation}\label{eq:psi3}
\left[\begin{array}{c|c} Z & 0 \\ \hline 0 & Z \end{array}\right] \left[\begin{array}{c} a \\ b \end{array} \right] \in \mathrm{Span}\left\{ \left[ \begin{array}{c} w_i e_{i+1} \\ e_{i+1} \end{array} \right] : i= 0, \dotsc , 2d \right\} = W.
\end{equation}
Mark that equation \eqref{eq:psi3} (or equivalently \eqref{eq:psi2}) does not automatically imply \eqref{eq:psi1}. Indeed, \eqref{eq:psi2} is satisfied if both members of the equation vanish, but these solutions are \textit{not admissible} for us since they provide no information on $\psi(z_i)$. Define 
\begin{equation*}
T = \left\{ \left[\begin{array}{c} a \\ b \end{array} \right] \in \overline K^{2d + 2} \text{ such that } \left[\begin{array}{c|c} Z & 0 \\ \hline 0 & Z \end{array}\right] \left[\begin{array}{c} a \\ b \end{array} \right] \in W \right\}.
\end{equation*}
Since $T$ is a subspace of $\overline K^{2d+ 2} $ defined by a linear system with coefficients in the subfield $\F$, it is spanned by vectors in $\F^{2d + 2}$, i.e.
\begin{equation}\label{eq:T}
T= \mathrm{Span}\{ v_1 , \dotsc, v_{\dim (T)} \} \quad \text{with } v_j \in \F^{2d + 2} \text{ for all } j.
\end{equation}
Moreover $T$ is nontrivial because $\phi$ satisfies \eqref{eq:psi1} and therefore its defining vector belongs to $T$. For every $i = 0 , \dotsc, 2d$ define the subspaces 
\begin{equation*}
H_i = \left\{ \left[\begin{array}{c} a \\ b \end{array} \right] \in \overline K^{2d + 2} \text{ such that } \sum_{j = 0}^{d} a_j z_i^j = \sum_{j=0}^{d} b_j z_i^j = 0 \right\} 
\end{equation*}
and put
\begin{equation*}
H= \bigcup_{i = 0}^{2d} H_i, \quad Y = T - H.
\end{equation*}
Then $Y$ is the set of admissible solutions and we claim that it contains the defining vector $(a_\phi, b_\phi)$ of $\phi$. Indeed, assume $(a_\phi,b_\phi) \in H$, then it belongs to some $H_i$, i.e. $z_i$ is a common root of the numerator and denominator of $\phi$, a contradiction to $\deg(\phi) = d$. We end the proof showing that $Y \cap \F^{2d + 2}$ is nonempty so that it suffices to take one of its elements as defining vector of $\psi$. This can be done by a simple argument on cardinalities. First, $T$ is not contained in any of the subspaces $H_i$ because $Y$ is nonempty; in particular $\dim(T \cap H_i) \leqslant \dim (T) -1$ for $i=0, \dotsc, 2d$. Furthermore the subspaces $T \cap H_i$ are spanned by vectors in $\F^{2d + 2}$, hence 
\begin{gather*}
\#(T \cap H_i \cap \F^{2d + 2}) \leqslant (\# \F)^{\dim(T) -1} \quad \text{and} \\
\#(T \cap \F^{2d+2} ) = (\# \F)^{\dim(T)} \quad \text{by \eqref{eq:T}.}
\end{gather*}
Therefore
\begin{equation*}
    \#(Y \cap \F^{2d + 2} ) \geqslant (\# \F)^{\dim(T)} - (2d + 1) (\# \F)^{\dim(T) - 1}.
\end{equation*}
Up to considering a finite extension of $\F$ we may assume $\# \F > 2d +1$ and thus $Y \cap \F^{2d + 2}$ is nonempty.
\end{proof}


\begin{cor}\label{Cdbound}
Let $\phi \colon \P^1 \rightarrow \P^1$ be a rational map defined over a global function field with one prime of bad reduction. Let $d \geqslant 1$ be the degree of the map and $\mathcal C \subseteq \P^1(K)$ a cycle for $\phi$. Then either $\# \mathcal C \leqslant 2d$, or after a change of coordinates $\phi$ is defined over a finite field.
\end{cor}
\begin{proof}
Fix the prime of bad reduction for $\phi$ at infinity. By Theorem \ref{T1} we may apply Proposition \ref{Prational} with $\mathcal P = \mathcal P' = \mathcal C$: condition (b) is satisfied because the set $\mathcal C$ is $\phi$-invariant.
\end{proof}

Similarly, we can prove a bound for the cardinality of finite orbits combining Proposition \ref{Prational} with Lemmas \ref{Lfixed} and \ref{Lfinite}.

\begin{cor}
Let $\phi \colon \P^1 \rightarrow \P^1$ be a rational map defined over a global function field with one prime of bad reduction. Let $d \geqslant 1$ be the degree of the map and $\mathcal A \subseteq \P^1(K)$ a finite orbit for $\phi$. Then either $\# \mathcal A \leqslant 6d^3 + 3$, or after a change of coordinates at least one of $\phi, \phi^2 $ and $\phi^3$ is defined over a finite field.
\end{cor}

\begin{proof}
Assume that there is no change of coordinates such that $\phi$ is defined over a finite field. Let $n$ denote the length of the cycle that is the periodic part of the orbit $\mathcal A$. In particular $n \leqslant 2d$ by Corollary \ref{Cdbound}. 

First assume $n \geqslant 4$, then by Lemma \ref{Lfinite} distinct points in $\Tail(\mathcal A)$ reduce to distinct points modulo $\p$ for every prime of good reduction. In other words $\delta_{\p} (P, Q) = 0 $ for every pair of distinct points in $\Tail(\mathcal A)$ and for all primes except for the one of bad reduction. Fix the bad prime at infinity and let $P$ be the unique tail point in $\mathcal A$ such that $\phi(P)$ is periodic. Then we can apply Proposition \ref{Prational} with $\mathcal P = \Tail(\mathcal A)$ and $\mathcal P' = \mathcal P - \{P\}$. Therefore $\# \Tail(\mathcal A) \leqslant 2d + 1$, thus $\# \mathcal A \leqslant 4d+1$. 

If $n \leqslant 3$, we consider the map $\phi^n$ instead. This gives a partition of $\mathcal A$ in $n$ distinct orbits for $\phi^n$, each one terminating in a fixed point. Let $\mathcal A'$ be one of these orbits and assume it contains $m$ points indexed as in \eqref{eq:finite}. Let $\mathcal P = \mathcal A' - \{P_{-m + 1} \} $, then $\mathcal P $ is a $\phi^n$-invariant set. Furthermore for every fixed prime of good reduction $\p$ the quantity $\delta_{\p}(P, Q) $ with $(P, Q) $ varying over the pairs of distinct points in $\mathcal P $ is constant by Lemma \ref{Lfixed}. By Proposition \ref{Prational} either $\# \mathcal P \leqslant 2 \deg(\phi^n)$, or $\phi^n$ is defined over a finite field (after some change of coordinates). Thus, assuming that neither of $\phi, \phi^2 $ and $\phi^3$ is defined over a finite field, one gets $\# \mathcal A' \leqslant 2d^3 + 1 $, hence $\# \mathcal A \leqslant 6d^3 + 3$.
\end{proof}


\end{document}